\documentclass[a4paper,11pt]{article}

\usepackage[title,titletoc,toc]{appendix}

\title{Riemann problem for a particle-fluid coupling}
\author{Nina Aguillon, Universit\'e Paris Sud}

\usepackage[english]{babel}
\usepackage{verbatim}
\usepackage{fancyvrb}
\usepackage[utf8x]{inputenc}
\usepackage{multicol,lipsum,float}
\usepackage[T1]{fontenc}
\usepackage{graphicx}
\usepackage{multicol}
\usepackage{graphics}
\usepackage{xcolor}
\usepackage{pgf}
\usepackage{amssymb}
\usepackage{amsmath}
\usepackage{amsfonts}
\usepackage{amssymb}
\usepackage{graphicx}
\usepackage{bm}
\usepackage{amsthm}
\usepackage{empheq}
\usepackage{enumerate}
\usepackage{textcomp}
\usepackage{mathrsfs}
\usepackage{caption}
\usepackage{stmaryrd}
\usepackage{url}
\usepackage{psfrag}
\usepackage{cases}
\usepackage[pdfauthor={Nina Aguillon}, bookmarksnumbered=true]{hyperref}

\newtheorem{thm}{Theorem}[section]
\newtheorem{lemma}[thm]{Lemma}
\newtheorem{prop}[thm]{Proposition}

\newtheorem{cor}[thm]{Corollary}

\theoremstyle{remark}
\newtheorem{rem}[thm]{Remark}

\theoremstyle{definition}
\newtheorem{defi}[thm]{Definition}
\newtheorem{ex}[thm]{Example}

\newcommand{\sign}{\text{sign}}

\newcommand{\R}{\mathbb{R}}

\newcommand{\eps}{\varepsilon}
\newcommand{\dis}{\displaystyle}

\begin{document}

\maketitle

\abstract{We present a model of coupling between a point wise particle and a compressible inviscid fluid following the Euler equations. The interaction between the fluid and the particle is achieved through a drag force. It writes as the product of a discontinuous function and a Dirac measure. After defining the solution, we solve the Riemann problem with a fixed particle for arbitrary initial data. We exhibit a set of conditions on the drag force under which there exists a unique self-similar solution. 

\bigskip

}

\tableofcontents

\section{Introduction}
In this paper we introduce and study a simple one-dimensional model of fluid-structure interaction. We consider a compressible and inviscid fluid, which will be governed by the isothermal Euler equations. At point $x$ and time $t$, it has velocity $u(t,x)$ and density $\rho(t,x)$. The particle is assumed to be pointwise, of mass $m$, and we denote by $h(t)$, $h'(t)$ and $h''(t)$ its position, velocity and acceleration at time $t$ . The interaction between the fluid and the particle is achieved through a drag force $D$, reflecting the fact that both tend to share the same velocity.  Our model writes
\begin{subequations} \label{eqcouple}
\begin{numcases}{}
     \partial_t \rho + \partial_x (\rho u) = 0  \label{MassConservation}, \\
    \partial_t (\rho u) + \partial_x \left( \rho u^{2}+p(\rho) \right) = - D(\rho, \rho(u-h'(t))) \delta_{h(t)}(x)   \label{ActionReaction},\\
     m h''(t)=    D(\rho(t,h(t)), \rho(t, h(t)) [u(t,h(t))-h'(t)]) \label{NewtonLaw}, \\
     (\rho(0,x), u(0,x))=(\rho_{0}(x), u_{0}(x)), \notag \\
 (h(0), h'(0))=(0,v_{0}) \notag,
\end{numcases}
\end{subequations}
where $p$ is the pressure law. We suppose that $D$ has the same sign as $ u-h'$ (and therefore than $\rho(u-h')$), which is natural for a drag force. As a result, Equation~\eqref{NewtonLaw} shows formally that the particle accelerates (respectively decelerates) when the fluid's velocity $u(t,h(t))$ at its position is larger (respectively smaller) than its own velocity $h'(t)$. We choose the isothermal pressure law $p(\rho)=c^{2} \rho$ to avoid vacuum related issues, but the results could probably be extended without major difficulty in the adiabatic case $p(\rho)= \rho^{\gamma}, 1<\gamma \leq 3$. This model is a generalization of the one first introduced in~\cite{LST08}, and later deeply investigate in~\cite{AS10},~\cite{ALST10} and ~\cite{ALST13}, in which the fluid followed a scalar Burgers equation.

In the past few years, the interaction between an incompressible and viscous fluid and rigid bodies has been widely studied. Many papers deal with the existence of weak or strong solutions (\cite{SMS02}, \cite{GHS00}, \cite{F03}, \cite{DE99} and~\cite{VZ03}). The matter of collision has also been extensively investigated (\cite{VZ06},  \cite{S04}, ~\cite{H07}, \cite{HT09}). Some other works consider compressible fluid or elastic bodies (\cite{BST12}, \cite{CDE05}, \cite{DE00} and~\cite{F03bis}). The model we study strongly differs from those works as the equation governing the fluid is inviscid. In this framework, the d'Alembert's paradox states that in an incompressible and inviscid fluid, with vanishing circulation, no drag force exerts on a body moving at a constant speed. Birds should therefore not be able to fly. An answer to this paradox is that, even at very high Reynolds number, the effect of viscosity cannot be neglected in a thin layer around the body. In our model, this paradox is somehow ignored, as we directly impose a drag force  $D$  between the fluid and the particle. According to Newton's law, the particle follows the ordinary differential equation~\eqref{NewtonLaw}
$$m h''(t)=  D.$$
The action-reaction principle is taken into account in equation~\eqref{ActionReaction} on the momentum of the Euler equation: the particle applies the force $-D$ on the fluid. We suppose that the interaction is local: it applies only at point $h(t)$. Equation~\eqref{MassConservation} ensures that the fluid mass is conserved. This approach proved to be successful in the toy model~\cite{LST08} with a Burgers' fluid. In particular, it allows collisions between two particles having different velocities, unlike in the viscous case~\cite{VZ06}. For example, a particle trapped in a shock (case $V$ of Lemma 5.5 in~\cite{LST08}) will collide with a particle placed in front of the first particle and sharing the velocity of the fluid. In~\cite{ALST10}, the reader can find numerical simulations of the drafting kissing tumbling phenomenon. This work is strongly related to~\cite{BCG14}, in which a coupling between a particle and the Euler equations is presented. However, the modelization is quite different, the particle being taken into account through conservation of mass and energy, while we enforce in the present work a drag force. Moreover, the nature of the theoretical results are different and complementary. In~\cite{BCG14}, a local in time existence to the Cauchy problem for the fully coupled system is proved for small subsonic data. In the present work, we consider the Riemann problem for a particle having constant velocity, without making any assumption on the data.
 
We emphasize the fact that in model~\eqref{eqcouple}, the particle and the fluid do not share the same velocity. We do not impose any no slip condition as in works presented above. It can be justified by saying that the particle is porous and allows some fluid to pass through. It constitutes the main difficulty of this model. Indeed, as shocks develop in finite time in hyperbolic systems like~\eqref{MassConservation}-\eqref{ActionReaction}, even with $D=0$, the velocity $u$ and the density $\rho$ of the fluid have no reason to be continuous along the particle trajectory $h$. A first consequence is that the source term in~\eqref{ActionReaction} is not well defined. A second one is that the ODE~\eqref{NewtonLaw} the particle satisfies must be considered in a weak sense.  This paper focuses on the Riemann problem for the uncoupled problem
\begin{equation} \label{eqfixe}
\begin{cases}
     \partial_t \rho + \partial_x (\rho u) = 0,   \\
    \partial_t (\rho u) + \partial_x \left( \rho u^{2}+p(\rho) \right) = -D(\rho, \rho(u-h'(t))) \delta_{vt}, \\
     \rho(0,x)=\rho_{L} \mathbf{1}_{x<0} + \rho_{R} \mathbf{1}_{x>0}, \\
     u(0,x)=u_{L} \mathbf{1}_{x<0} + u_{R} \mathbf{1}_{x>0},
\end{cases}
\end{equation}
where the particle has a constant speed $v$. The difficulties arising from the coupling between an ODE and a PDE disappear, but the key point of the nonconservative source term remains. Another main difficulty in the analysis of~\eqref{eqfixe} is that the Dirac measure in the source term corresponds to a linearly degenerate field and our system is not hyperbolic. This may lead to resonance phenomena when two families of waves interact. Near resonance, Riemann problems with such source terms have been investigated in a conservative framework in~\cite{IT92} and later extended in the nonconservative framework in~\cite{GLF04}. Away from resonance, the particle trajectory can be treated as a noncharacteristic boundary (see~\cite{BCG10},~\cite{BCG12} and ~\cite{BCG14}). Our contribution is that, unlike in those frameworks, we solve the Riemann problem for all choices of parameters $(\rho_{L},u_{L})$, $(\rho_{R},u_{R})$ and $v$, without making any assumptions on their smallness or their resonant character.

Let us outline the organization of the paper and sketch the main results. The first section is entirely devoted to the definition of the solutions of~\eqref{eqcouple}. We exhibit an entropy condition that takes into account the particle. Then we give a rigorous definition of the nonconservative product
$$  D(\rho, \rho(u-h'(t) ) ) \delta_{h(t)}  $$ 
based on a thickening of the particle. Replacing the Dirac measure by one of its approximation, it appears that the density and velocity of the fluid at the entry of the particle and at its exit are always linked by the same relations. These relations are independent of the size and the shape of the thickened particle. This allows us to see the particle as an interface, through which those relations are imposed. They link the quantities $u(t,h(t)^{\pm})$, $\rho(t,h(t)^{\pm})$ and $h'(t)$. A first relation states that the quantity $\alpha := \rho (u-h')$ is constant across the particle (this is why we express $D$ as a function of $\rho$ and $\alpha$). This is equivalent to the conservation of the fluid's mass through the particle. Another relation describes the influence of the particle on the flow and depends on $D$.
When the drag force is $D(\rho, \alpha)= \alpha$ and the particle is motionless, it expresses that the loss of charge $ \rho u^{2}+ c^{2} \rho$ through the particle is proportional to the mass flow $\rho u$. The second section is devoted to the solution of the Riemann problem~\eqref{eqfixe} for a particle moving at constant speed $v$. In Theorem~\ref{allfriction2}, we exhibit a two conditions on the drag force $D$ that imply that~\eqref{eqfixe} has a unique self-similar entropy solution. The case of subsonic and supersonic initial datum are treated separately in Propositions~\ref{RSubSub} and~\ref{RSupSup}. In Subsection~\ref{SAsymptotics} we describe the two natural asymptotics when the drag force vanishes or becomes very large. Eventually in Section~\ref{SMultipleSolutions}, we discuss the case where the hypothesis of Theorem~\ref{allfriction2} are not fulfilled. We recover, in some particular cases, the existence of up to three solutions, as shown in a general framework in~\cite{IT92} and~\cite{GLF04}, and well known for fluid in a nozzle with discontinuous cross-section~\cite{TLF2} and for the shallow water equation with discontinuous decreasing topography~\cite{TLF1}.

\paragraph{Acknowledgments}
The author warmly thanks Frédéric Lagoutière for his advice and support during the achievement of this work.

\section{Definition of the solutions}
This section is devoted to the definition of the solution of the coupled system~\eqref{eqcouple}. The isothermal Euler equations are inviscid, so $\rho$ and $u$ can be discontinuous along the particle's trajectory $h$ and the product  $D(\rho, \rho (h'-u)) \delta_{h}(x)$ does not make sense. Following~\cite{LST08}, we consider two different regularizations in Sections~\ref{Sentropineq} and~\ref{SNCproduct}.  The first one consists in adding a vanishing viscosity to the equation. Passing to the non-regularized limit we deduce an entropy inequality for~\eqref{eqcouple}.  The second regularization is a thickening of the particle, which yields to an intrinsic definition of the non-conservative source term as an interface. 
 

\subsection{Entropy inequality} \label{Sentropineq}
Let us first focus on the following classical regularization of problem~\eqref{eqcouple}, where we add a vanishing viscosity to the Euler equation:
\begin{equation} \label{eqvisco}
 \left\{
 \begin{array}{ll}
 \partial_t \rho^{\eps} + \partial_x q^{\eps}   = \eps \partial_{xx} \rho^{\eps}, \\
 \partial_t q + \partial_x \left( \frac{(q^{\eps})^2}{\rho^{\eps}} +c^2 \rho^{\eps} \right)  = - D(\rho^{\eps},\alpha^{\eps}) \delta_{h(t)} + \eps \partial_{xx} q^{\eps},\\
 m h''(t)=  D(\rho^{\eps}(t,h(t)),\alpha^{\eps}(t,h(t)) ),\\
 (\rho^{\eps}(0,x), q^{\eps}(0,x))=(\rho_{0}^{\eps}(x), q^{\eps}_{0}(x)), \\
 (h(0), h'(0))=(h_{0},v_{0}).
\end{array} 
\right. 
\end{equation}
Here, $q= \rho u$ denotes the momentum of the fluid, $\alpha$ denotes the quantity $\rho(u-h')$ and we only assume that the drag force $D$ has the same sign than $\alpha$ (and hence than $u-h'$). In~\cite{LST08} and~\cite{D02} it is proven that the system
\begin{equation*}
 \left\{
 \begin{array}{ll}
 \partial_t u^{\eps} + \partial_x \frac{(u^{\eps})^{2}}{2}   = \eps \partial_{xx} u^{\eps} - \lambda (u^{\eps} -h') \delta_{h(t)}(x) \\
 m h''(t)=  \lambda (u^{\eps} -h'),\\
 u^{\eps}(0,x)= u^{0}(x), \\
 (h(0), h'(0))=(h_{0},v_{0}).
\end{array} 
\right. 
\end{equation*}
admits regular a solution when $u^{0}$ is regular. 
 In order to derive an entropy inequality for our fluid particle coupling~\eqref{eqcouple}, we assume that  if the initial data $(\rho^{0}, u^{0})$ are smooth and that the solutions of its regularization~\eqref{eqvisco} are also smooth.
Let $E(\rho,q)$ and $G(\rho,q)$ be a flux-entropy flux pair, with $E$ convex, such that $ \dis \partial_{q} E$ is a function of $u= \frac{q}{\rho}$ denoted by $J$. 
\begin{ex} \label{entropy}
The usual entropy-entropy flux pair
$$E(\rho,q)=\frac{q^{2}}{2\rho}+c^{2} \rho \log(\rho) \ \ \ \textrm{ and } \ \ \ G(\rho,q)= \frac{q}{\rho} (E(\rho,q)+c^{2} \rho)$$
fulfills this assumption: we have $J(u)=u$.
\end{ex}
For the sake of simplicity we introduce
$$ U_{0}^{\eps}=(\rho_{0}^{\eps},q_{0}^{\eps}), \ \ U^{\eps}=(\rho^{\eps},q^{\eps}) \ \ \ \textrm{ and } \ \ \ F(U^{\eps})=\left(q^{\eps}, \frac{(q^{\eps})^{2}}{\rho^{\eps} } +c^{2} \rho^{\eps} \right). $$
Let $\Phi \in \mathcal{C}_{0}^{\infty} (\R_{+}\times \R)$ be a \emph{non-negative} smooth function, and multiply the first equation of~\eqref{eqvisco} by $\Phi \partial_{\rho} E $ and the second equation by $\Phi \partial_{q} E = \Phi J$. Let us add the two equations and integrate over $\R_{+} \times \R$. We obtain
\begin{equation} \label{EIF}
 \begin{aligned}
\iint_{\R_{+} \times \R}  \Phi \nabla_{(\rho,q)} E(U^{\eps}) \cdot \partial_{t} U^{\eps}  \, dt \, dx + \iint_{\R_{+} \times \R} \Phi  \nabla_{(\rho,q)} E(U^{\eps}) \cdot \partial_{x} F(U^{\eps})  \, dt \, dx  \\
 -\eps \iint_{\R_{+} \times \R}  \Phi  \nabla_{(\rho,q)} E(U^{\eps}) \cdot \Delta U^{\eps} \, dt \, dx = -  \int_{\R_{+}} \left[ D(\rho^{\eps}, \alpha^{\eps} ) \Phi J( u^{\eps}) \right] (t,h(t)) dt .
\end{aligned}
\end{equation}
We first treat the left hand side of~\eqref{EIF} by integrating by parts. The first term gives
$$ 
\begin{aligned}
 \iint_{\R_{+} \times \R}  \Phi \nabla_{(\rho,q)} E(U^{\eps}) \cdot \partial_{t} U^{\eps}  \, dt \, dx & = \iint_{\R_{+} \times \R}  \Phi \partial_{t} E(U^{\eps}) \, dt \, dx \\
	&= - \iint_{\R_{+} \times \R}  E(U^{\eps}) \partial_{t} \Phi \, dt \, dx - \int_{\R} \Phi(0,\cdot) E(U_{0}^{\eps}) dx 
\end{aligned}
$$
and the second term yields
$$ 
\begin{aligned}
 \iint_{\R_{+} \times \R} \Phi \nabla_{(\rho,q)} E (U^{\eps}) \cdot \partial_{x} F(U^{\eps})  \, dt \, dx  &=  \iint_{\R_{+} \times \R} \Phi \nabla_{(\rho,q)}E(U^{\eps})  \cdot \left[ D F(U^{\eps}) \partial_{x} U^{\eps} \right] \, dt \, dx \\
 	&= \iint_{\R_{+} \times \R} \Phi  \nabla_{(\rho,q)} G(U^{\eps}) \cdot \partial_{x} U^{\eps}  \, dt \, dx \\ 
	&= \iint_{\R_{+} \times \R} \Phi  \partial_{x} G(U^{\eps})  \, dt \, dx \\
	&= - \iint_{\R_{+} \times \R}  G(U^{\eps}) \partial_{x} \Phi  \, dt \, dx .
\end{aligned}
$$
Let us now tackle the third term. We have
$$ 
\begin{aligned}
 \nabla_{(\rho,q)} E(U^{\eps}) \cdot \Delta U^{\eps} &= \sum_{i=1}^{2} \partial_{i} E(U^{\eps}) \partial_{xx} U^{\eps}_{i} \\
	&= \sum_{i=1}^{2} \left[ \partial_{x} ( \partial_{i} E(U^{\eps}) \partial_{x} U^{\eps}_{i}) - \left( \sum_{j=1}^{2}\partial_{ji}E(U^{\eps}) \partial_{x} U_{j}^{\eps} \right) \partial_{x} U^{\eps}_{i} \right] \\
	&= \partial_{xx}E(U^{\eps}) - \sum_{i=1}^{2} \sum_{j=1}^{2} ( \partial_{x} U_{j}^{\eps} ) (\partial_{ji}E(U^{\eps})) ( \partial_{x} U_{i}^{\eps} ),
\end{aligned}
$$
thus we obtain
$$ 
\begin{aligned}
 -\eps \iint_{\R_{+} \times \R}  \Phi \nabla_{(\rho,q)} E(U^{\eps}) \cdot \Delta U^{\eps} \, dt \, dx = &  -\eps \iint_{\R_{+} \times \R}  \Phi E(U^{\eps}) \partial_{xx} \Phi  \, dt \, dx \\
  	&+  \eps \iint_{\R_{+} \times \R}  \Phi \sum_{i=1}^{2} \sum_{j=1}^{2} ( \partial_{x} U_{j}^{\eps} ) \partial_{ij}E(U^{\eps}) ( \partial_{x} U_{i}^{\eps} ) \, dt \, dx.
\end{aligned}
$$
Remark that as $\Phi$ is non-negative and $E$ is convex, the last term is non-negative. We now focus on the right hand side of~\eqref{EI}. Let us multiply the ODE in the third equation of~\eqref{eqvisco} by $J(h'(t))\Phi(t,h(t))$. We have
$$ \int_{\R_{+}} - m h''(t)  J(h'(t))\Phi(t,h(t)) +  D(\rho^{\eps},\alpha^{\eps}) (t,h(t)) J(h'(t)) \Phi(t,h(t)) dt = 0, $$
which reads, with $P$ an antiderivative of $J$,
$$ \int_{\R_{+}} - m [P(h'(t))]'\Phi(t,h(t)) +  D(\rho^{\eps},\alpha^{\eps}) (t,h(t)) J(h'(t))\Phi(t,h(t)) dt = 0. $$
We can therefore replace the right hand side of~\eqref{EI} by
$$
 \int_{\R_{+}} m P(h'(t)) \partial_{t}(\Phi(t,h(t))) dt + \int_{\R_{+}} D(\rho^{\eps},\alpha^{\eps}) \left( J(h'(t))- J(u^{\eps}) \right)\Phi(h) dt + m P(v_{0}) \Phi(0,h_{0}).
$$
The function $J(u)=\partial_{q} E(1,u)$ is nondecreasing as the restriction of $E$ to the line $\rho=1$ is convex. Moreover,  $D$ has the same sign as $u^{\eps}-h'$. Thus the second term is non-positive. To conclude, we add the different terms and drop the two non-positive ones to obtain
$$
\begin{aligned}
 \iint_{\R_{+} \times \R}  E(U^{\eps}) \partial_{t} \Phi \, dt \, dx + \iint_{\R_{+} \times \R}  G(U^{\eps}) \partial_{x} \Phi  \, dt \, dx & +  \int_{\R_{+}} m P(h'(t)) \partial_{t}(\Phi(t,h(t))) dt \\
+ \int_{\R} \Phi(0,\cdot) E(U_{0}^{\eps}) \, dx + m P(v_{0}) \Phi(0,h_{0}) & \geq -\eps \iint_{\R_{+} \times \R}  \Phi E(U^{\eps}) \partial_{xx} \Phi  \, dt \, dx .
\end{aligned}
$$
Last, we formally pass to the limit as $\eps \longrightarrow 0$ and get the following entropy inequality for the coupled problem~\eqref{eqcouple}:
\begin{equation} \label{EntropyInequality}
 \begin{aligned}
 \iint_{\R_{+} \times \R}  E(U) \partial_{t} \Phi \, dt \, dx + \iint_{\R_{+} \times \R}  G(U) \partial_{x} \Phi  \, dt \, dx  +  \int_{\R_{+}} m P(h'(t)) \partial_{t}(\Phi(t,h(t))) dt \\
 + \int_{\R} \Phi(0,\cdot) E(U_{0}) \, dx + m P(v_{0}) \Phi(0,h_{0}) \geq 0.
\end{aligned}
\end{equation}
\begin{rem}
 When the test function $\Phi$ is supported on $\{ (t,x), \, x>h(t) \}$ or on $\{ (t,x), \, x<h(t) \}$, the inequality~\eqref{EntropyInequality} reduces to the classical entropy inequality for the isothermal Euler equation without source term.
\end{rem}

\subsection{How to handle the nonconservative product} \label{SNCproduct}
In this section we assume that the particle trajectory $h$ is given, and more precisely that it moves at constant speed $v$: $h(t)=vt$. 
We denote by $H$ the Heaviside function $H(x)= \mathbf{1}_{x>0}$. Introducing the momentum $q=\rho u$ and the new unknown $w(t,x)= H(x-vt)$ allows us to rewrite the system~\eqref{MassConservation}-\eqref{ActionReaction} in the framework of hyperbolic equation,
\begin{equation}
 \begin{cases}
 \partial_t \rho + \partial_x q = 0,   \\
  \partial_t q + \partial_x \left( \frac{q^{2}}{\rho}+c^2 \rho \right) + D(\rho, \alpha) \partial_{x} w = 0,\\
  \partial_{t} w + v \partial_{x} w=0.
\end{cases}
\end{equation}
Its quasilinear form is
\begin{equation}
\partial_{t} 
\begin{pmatrix}
 \rho \\ q \\ w 
\end{pmatrix}
+ 
\begin{pmatrix}
 0 & 1 & 0 \\
 c^{2}-u^{2} & 2u &D(\rho, \alpha) \\
 0 & 0 & v
\end{pmatrix}
\partial_{x}
\begin{pmatrix}
 \rho \\ q \\ w 
\end{pmatrix}
=0.
\end{equation}
The eigenvalues of the Jacobian matrix are $u+c$, $u-c$ and $v$. This system is strictly hyperbolic whenever $u \neq v \pm c$. In the resonant case $u \pm c = v$, the matrix cannot be put in a diagonal form. Such resonant systems have been studied in~\cite{GLF04} and~\cite{IT92}.  However, with this source term, our system does not fall neither in the framework of~\cite{IT92}, because it is not conservative, neither in the framework of~\cite{GLF04}, because one of the hypothesis on the source term (namely 1.7) is not satisfied when the drag force depends only on $\alpha$. Following~\cite{LST08},~\cite{GLF04},~\cite{IT92},~\cite{SV03} and~\cite{CLS04}, we use a thickening of the particle to define the nonconservative product.
 Let $H^{\eps}$ be an approximation of the Heaviside function such that:
\begin{itemize}
 \item $H^{\eps} \in \mathcal{C}^{0}(\R) \cap \mathcal{C}^{1}((-\eps/2, \eps/2))$;
 \item $H^{\eps}$ is nondecreasing;
 \item $H^{\eps}(x)=0$ if $x \leq -\eps/2$ and $H^{\eps}(x)=1$ if $x \geq \eps/2$.
\end{itemize}
We replace the Dirac measure by its regularization $(H^{\eps})'$ to obtain the regularized system
\begin{equation} \label{eqregul}
 \left\{
 \begin{array}{ll}
  \partial_t \rho^{\eps} + \partial_x q^{\eps} &= 0; \\
 \partial_t q^{\eps} + \partial_x \left(  \frac{ (q^{\eps})^2}{\rho^{\eps}} +c^2 \rho^{\eps} \right) &= -D(\rho^{\eps}, \alpha^{\eps}) (H^{\eps})'(x-vt).
\end{array} 
\right. 
\end{equation}
We are interesting in what is happening inside the particle. In the spirit of traveling waves, we look for solution only depending on $x-vt$.  With such a regularized source term, the values of the solutions at the extremities of the particle depend neither on the size of the thickened particle $\eps$, nor on the choice of the regularization $H^{\eps}$ (satisfying the three hypotheses above). This allows us to define the source term $D(\rho, \alpha) \delta_{h(t)}$ as an interface condition.

In the sequel, for $\alpha \neq 0$,  we denote by $F_{\alpha}$ the function
$$ F_{\alpha}(\rho) = \int_{\frac{|\alpha|}{c}} ^{\rho} \frac{1}{|D(r, \alpha)|} \left( - \frac{\alpha^{2}}{r} + c^{2} \right) \, dr.$$
Remark that $F_{\alpha}$ decreasing on $\left(0, \frac{|\alpha|}{c}\right)$ and increasing on $\left( \frac{|\alpha|}{c}, + \infty \right)$.
\begin{lemma} \label{lemmagerm}
Let $(\rho^{\eps}, q^{\eps})$ be a piecewise $\mathcal{C}^{1}$ solution of~\eqref{eqregul}  that only depends on $\xi= x-vt$ and defined for $\xi$ in $[-\eps/2, \eps/2]$. Then, on every interval $\mathcal{I}$ where the solution is smooth, the quantity $\alpha^{\eps} =q^{\eps}- v \rho^{\eps}$ remains constant. If $\alpha^{\eps}=0$ on $\mathcal{I}$, the density also remains constant, while if $\alpha^{\eps} \neq 0$, the evolution of $\rho^{\eps}$ on $\mathcal{I}$ is given by
\begin{equation} \label{decF}
 \left( F_{\alpha^{\eps}}(\rho^{\eps}(\xi)) \right)'= - \sign(\alpha^{\eps})   (H^{\eps})'.
\end{equation}
If the solution is discontinuous at a point $\xi^{0} \in (-\eps/2, \eps/2)$, then
\begin{equation} \label{RH}
\begin{cases}
  \alpha^{\eps}(\xi^{0}_{-})=\alpha^{\eps}(\xi^{0}_{+}):= \alpha^{\eps}(\xi^{0}) \, ; \\
 \dis  \left(\frac{(\alpha^{\eps}(\xi^{0}))^{2}}{\rho^{\eps}(\xi^{0}_{-})} +c^2 \rho^{\eps}(\xi^{0}_{-}) \right) - \left(\frac{(\alpha^{\eps}(\xi^{0}))^{2}}{\rho^{\eps}(\xi^{0}_{+})} +c^2 \rho^{\eps}(\xi^{0}_{+}) \right) = 0\end{cases}\end{equation}
\end{lemma}

\begin{proof}
Let $\rho^{\eps}(x-vt)$ and $q^{\eps}(x-vt)$ be a piecewise $\mathcal{C}^{1}$ solution  of~\eqref{eqregul}, only depending on $\xi=x-vt$. If the solution is smooth on the interval $\mathcal{I}$, it satisfies the following equations:
\begin{equation} \label{eqregulCI}
\left\{
 \begin{array}{ll}
 -v (\rho^{\eps})' + (q^{\eps})' &= 0, \\
-v  (q^{\eps})' + \left(\frac{(q^{\eps})^2}{\rho^{\eps}} +c^2 \rho^{\eps} \right)' &= -D(\rho^{\eps}, \alpha^{\eps}) (H^{\eps})'(x-vt).
\end{array} 
\right.
\end{equation}
The first equation of \eqref{eqregulCI} directly gives that $\alpha^{\eps}$ remains constant on $\mathcal{I}$. Replacing $q^{\eps}$ by $\alpha^{\eps}+v \rho^{\eps}$ in the second line of~\eqref{eqregulCI} yields
$$
-v^{2} (\rho^{\eps})' + \left(\frac{(\alpha^{\eps})^2+ 2 \alpha^{\eps} v \rho^{\eps} + v^{2} (\rho^{\eps})^{2} }{\rho^{\eps}} +c^2 \rho^{\eps} \right)' = - D(\rho^{\eps}, \alpha^{\eps})   (H^{\eps})' .$$
As $\alpha^{\eps}$ and $v$ are constant,  this expression simplifies in
$$  \left(- \frac{(\alpha^{\eps})^{2}}{(\rho^{\eps})^{2}}+ c^{2}  \right) (\rho^{\eps})'= - D(\rho^{\eps}, \alpha^{\eps}) (H^{\eps})' ,$$
which rewrites, by definition of $F_{\alpha}$,
\begin{equation*} 
 \left[ F_{\alpha^{\eps}}(\rho^{\eps}(\xi)) \right]'=\frac{1}{|D(\rho^{\eps}, \alpha^{\eps})|} \left(-\frac{(\alpha^{\eps})^2}{(\rho^{\eps}) ^{2}} +c^2 \right) ( \rho^{\eps})'  = - \sign(\alpha^{\eps})   (H^{\eps})'.
\end{equation*}
On the other hand, if $(\rho^{\eps},q^{\eps})$ has a discontinuity at a point $\xi^{0} \in (-\eps/2,\eps/2)$, it verifies the two relations:
\begin{equation*}
\begin{cases}
  q^{\eps}(\xi^{0}_{+})-q^{\eps}(\xi^{0}_{-})= v (\rho^{\eps}(\xi^{0}_{+})-\rho^{\eps}(\xi^{0}_{-})), \\
  \dis \left( \frac{(q^{\eps}(\xi^{0}_{+}))^{2}}{\rho^{\eps}(\xi^{0}_{+})} + c^{2} \rho^{\eps}(\xi^{0}_{+}) \right) - \left( \frac{(q^{\eps}(\xi^{0}_{-}))^{2}}{\rho^{\eps}(\xi^{0}_{-})} + c^{2} \rho^{\eps}(\xi^{0}_{-}) \right) = v (q^{\eps}(\xi^{0}_{+}) - q^{\eps}(\xi^{0}_{-}) ),
\end{cases}
\end{equation*}
and we obtain the result~\eqref{RH} by introducing the conserved quantity $\alpha(\xi^{0})=q^{\eps}(\xi^{0}_{-})-v \rho^{\eps}(\xi^{0}_{-})= q^{\eps}(\xi^{0}_{+})-v \rho^{\eps}(\xi^{0}_{+})$.
\end{proof}

\begin{rem}
 The relations~\eqref{RH} are nothing but the Rankine-Hugoniot relations for a shock having speed $v$ in the isothermal Euler equations. The lemma below states that entropy shocks only link supersonic states to subsonic states (from left to right if $\alpha>0$, from right to left if $\alpha<0$).
 \end{rem}

\begin{lemma} \label{entropyshocks}
 The shock corresponding to the Rankine-Hugoniot relations~\eqref{RH} is an entropy satisfying shock in the Euler equations (without source term) for the entropy-entropy flux pair of Example~\ref{entropy} if and only if $\alpha^{\eps}(\xi^{0})>0$ and $\alpha^{\eps}(\xi^{0})>c \rho^{\eps}(\xi^{0}_{-})$ or if $\alpha^{\eps}(\xi^{0})<0$ and $\alpha^{\eps}(\xi^{0})<-c \rho^{\eps}(\xi^{0}_{+})$.
\end{lemma}

\begin{proof}
 Suppose that $\alpha^{\eps}(\xi^{0})$ is positive.
 If $c \rho^{\eps}(\xi^{0}_{-})> \alpha^{\eps}$ then $\rho^{\eps}(\xi^{0}_{-})>\rho^{\eps}(\xi^{0}_{+})$. If the shock was a Lax shock, it should be a $2$-shock. Therefore we should have
 $$ v = \frac{q^{\eps}(\xi^{0}_{+})}{\rho^{\eps}(\xi^{0}_{+})}+ c \sqrt{\frac{\rho^{\eps}(\xi^{0}_{-})}{\rho^{\eps}(\xi^{0}_{+})}}, $$
which rewrites
$$ c \sqrt{\frac{\rho^{\eps}(\xi^{0}_{-})}{\rho^{\eps}(\xi^{0}_{+})}} = -\frac{\alpha^{\eps}}{\rho^{\eps}(\xi^{0}_{+})}, $$
and contradicts the fact that $\alpha^{\eps}>0$.  On the other hand if $c \rho^{\eps}(\xi^{0}_{-})< \alpha^{\eps}$, we have that $\rho^{\eps}(\xi^{0}_{-})<\rho^{\eps}(\xi^{0}_{+})$, and the discontinuity should be a $1$-shock. We obtain 
 $$ c \sqrt{\frac{\rho^{\eps}(\xi^{0}_{-})}{\rho^{\eps}(\xi^{0}_{+})}}=  \frac{\alpha^{\eps}}{\rho^{\eps}(\xi^{0}_{+})}, $$
which does not contradict the sign of $\alpha^{\eps}$. Similarly, we obtain that if $\alpha$ is negative, the jump is an entropy satisfying shock if and only if $ c\rho^{\eps}(\xi^{0}_{+})< |\alpha^{\eps}|$.  

\end{proof}

\begin{defi} \label{defgermgen}
 The germ at speed $v$ is the set $\mathcal{G}_{D}(v)$ containing all the pairs $((\rho_{-},q_{-}), (\rho_{+},q_{+}))$ of $(\R_{+}^{*} \times \R)^{2}$ verifying the two following relations.
\begin{enumerate} 
 \item First,
 $$ q_{-}-v \rho_{-}=q_{+}-v \rho_{+}. $$
 We denote by $\alpha$ this quantity.
 \item Second,
\begin{itemize}
 \item either $\alpha=0$, $\rho_{-}= \rho_{+}$ and $q_{-}=q_{+}$;
 \item or $ 0< \alpha $ , $ (\frac{\alpha}{c} -\rho_{+} )  (\frac{\alpha}{c} -\rho_{-} ) \geq 0 $ and
 $$ F_{\alpha}(\rho_{-})- F_{\alpha}(\rho_{+})=  1; $$
  \item  or $\alpha < 0$, $ (\frac{|\alpha|}{c} -\rho_{+} )  (\frac{|\alpha|}{c} -\rho_{-} ) \geq 0 $ and
 $$ F_{\alpha}(\rho_{+}) -F_{\alpha}(\rho_{-})=  1; $$
 \item or $ c \rho_{-} < \alpha$, $\rho_{+}  \geq \frac{\alpha}{c}$, and there exists $\rho \in (\rho_{-}, \frac{\alpha}{c})$ and $\theta \in [0,1]$ such that 
 $$ 
\begin{cases}
 F_{\alpha}(\rho_{-})- F_{\alpha}(\rho) &=   \theta, \\
 F_{\alpha}(\frac{\alpha^{2}}{c^{2} \rho})- F_{\alpha}(\rho_{+}) &=  (1-\theta); \\
\end{cases}
$$
\item or $   \alpha < -c \rho_{+}$, $\rho_{-}  \geq \frac{|\alpha|}{c}$ and there exists $\rho \in (\rho_{+}, \frac{|\alpha|}{c})$ and $\theta \in [0,1]$ such that
 $$ 
\begin{cases}
 F_{\alpha}(\rho_{+})- F_{\alpha}(\rho) &=  \theta, \\
 F_{\alpha}(\frac{\alpha^{2}}{c^{2}\rho})- F_{\alpha}(\rho_{-}) &=   (1-\theta). \\
\end{cases}
$$
\end{itemize}
\end{enumerate}
\end{defi}

\begin{thm} \label{Tdefgermgen}
Suppose that $((\rho_{-},q_{-}), (\rho_{+},q_{+}))$ belongs to $\mathcal{G}_{D}(v)$. Then for all positive $\eps$ and for all regularization $H^{\eps}$ fulfilling the hypothesis of~\eqref{eqregul}, there exists a piecewise $\mathcal{C}^{1}$ entropy solution of~\eqref{eqregul} only depending  on $\xi=x-vt$, such that $(q^{\eps}(-\eps/2), \rho^{\eps}(-\eps/2))= (q_{-}, \rho_{-})$ and $(q^{\eps}(\eps/2), \rho^{\eps}(\eps/2))= (q_{+}, \rho_{+})$. By entropy solution, we mean that each discontinuity in the solution corresponds to a entropy satisfying shock in the Euler equations.

Conversly, if $(\rho_{-},q_{-})$ and $(\rho_{+},q_{+})$ are the values in $-\eps/2$ and $\eps/2$ of such a solution of~\eqref{eqregul}, then they verify the two relations stated above.
\end{thm}

\begin{proof} Let $(\rho^{\eps}, q^{\eps})$ be a solution of~\eqref{eqregul} which depends only on $\xi= x-vt$, is piecewise $\mathcal{C}^{1}$ and whose discontinuities are entropy shocks. A straightforward consequence of Lemma~\ref{lemmagerm} is that the quantity $\alpha^{\eps}$ (equals to $\rho^{\eps}(u^{\eps}-v)$) is constant on the whole interval $[-\eps/2, \eps/2]$. In the sequel we suppose that $\alpha^{\eps}$ is positive. In that case the fluid's velocity $u^{\eps}$ is everywhere larger than the particle's velocity $v$ and the ``entry'' of the particle is on its right at $\xi= -\eps/2$. We  fix a state $(\rho_{-},q_{-}) \in \R_{+}^{*} \times \R$ at the entry of the particle, and look for the $(\rho_{+},q_{+}) \in \R_{+}^{*} \times \R$ at its exit. The reasoning is the same with $\alpha^{\eps}<0$ (and trivial if $\alpha^{\eps}=0$), but the entry of the particle is on its right and it is more convenient to fix the state $(\rho_{+},q_{+})$. 

There is only one solution which is smooth on the entire interval $[-\eps/2, \eps/2]$. We integrate~\eqref{decF} on this interval to obtain
$$ F_{\alpha^{\eps}}(\rho_{-})- F_{\alpha^{\eps}}(\rho_{+}) = 1. $$
As depicted on Figure~\ref{Fgerm}, the function  $F_{\alpha^{\eps}}$ decreases on $(0, \frac{\alpha^{\eps}}{c})$ and increases on $( \frac{\alpha^{\eps}}{c}, +\infty)$. Its minimum is reached for $ \frac{\alpha^{\eps}}{c}$. Moreover, the regularization $H^{\eps}$ of the Heaviside function is increasing. As a consequence, $\xi \mapsto F_{\alpha^{\eps}}(\rho^{\eps}(\xi))$ decreases, and $\rho^{\eps}$ cannot cross continuously $\alpha^{\eps}/c$. On its interval of smoothness, a solution of~\eqref{eqregul} is always subsonic (i.e $|u^{\eps}-v| \leq c$ or equivalently $c \rho^{\eps} \geq \alpha^{\eps}$) or always supersonic (i.e $|u^{\eps}-v| \geq c$ or equivalently $c \rho^{\eps} \leq \alpha^{\eps}$). This explains the condition $ (\frac{\alpha}{c} -\rho_{+} )  (\frac{\alpha}{c} -\rho_{-} ) \geq 0 $ in the second point of $2$ in Definition~\ref{defgermgen}.

On the other hand by Lemma~\ref{entropyshocks}, a discontinuity at a point $\xi^{0}$ is entropy satisfying if and only if $\alpha^{\eps}(\xi^{0})>c \rho^{\eps}(\xi^{0}_{-})$. Therefore a solution has no discontinuity if $\alpha_{-} \leq c \rho_{-}$ and has at most one discontinuity if $\alpha_{-} > c \rho_{-}$. We focus on this last case. The solution is smooth  on $(-\eps/2, \xi^{0})$ and~\eqref{decF} yields:
$$ F_{\alpha}(\rho_{-})- F_{\alpha}(\rho^{\eps}(\xi^{0}_{-}))= H^{\eps}(\xi^{0}). $$
There is a shock in $\xi^{0}$, and Rankine-Hugoniot relations~\eqref{RH} imply that
$$\rho^{\eps}(\xi^{0}_{-}) \rho^{\eps}(\xi^{0}_{+}) = \frac{(\alpha^{\eps})^{2}}{c^{2}}. $$
In particular, $c \rho^{\eps}(\xi^{0}_{+}) \geq \alpha^{\eps}$ and there is no shock after $\xi^{0}$. We integrate~\eqref{decF} on $(\xi^{0}, \eps/2)$ to get
$$ F_{\alpha}(\rho^{\eps}(\xi^{0}_{+}))- F_{\alpha}(\rho^{\eps}_{+})= 1-H^{\eps}(\xi^{0}). $$
We obtain the third point with $\rho= \rho^{\eps}(\xi^{0}_{-})$ and $\theta=H^{\eps}(\xi^{0})$. The two types of solutions, continuous everywhere or containing a single entropy shock, are described on Figure~\ref{Fgerm}.
\begin{psfrags} 
 \psfrag{al}{$1$}
  \psfrag{rho}{$\rho$}
  \psfrag{t}{$\theta$}
  \psfrag{1-t}{$1-\theta$}
  \psfrag{2ac}{$2 \alpha c$}
  \psfrag{2ac+al}{$2 \alpha c+ \lambda \alpha$}
  \psfrag{a/c}{ $\frac{|\alpha|}{c} $}
  \psfrag{1C}[][][1][-4.5]{ entropy shock}
  \psfrag{NEC}[][][1][8]{ non entropy shock}
  \psfrag{rho ->F}{$\rho \mapsto F_{\alpha}(\rho)$}
 \begin{figure}[H]
 \centering
 \includegraphics[width=16cm]{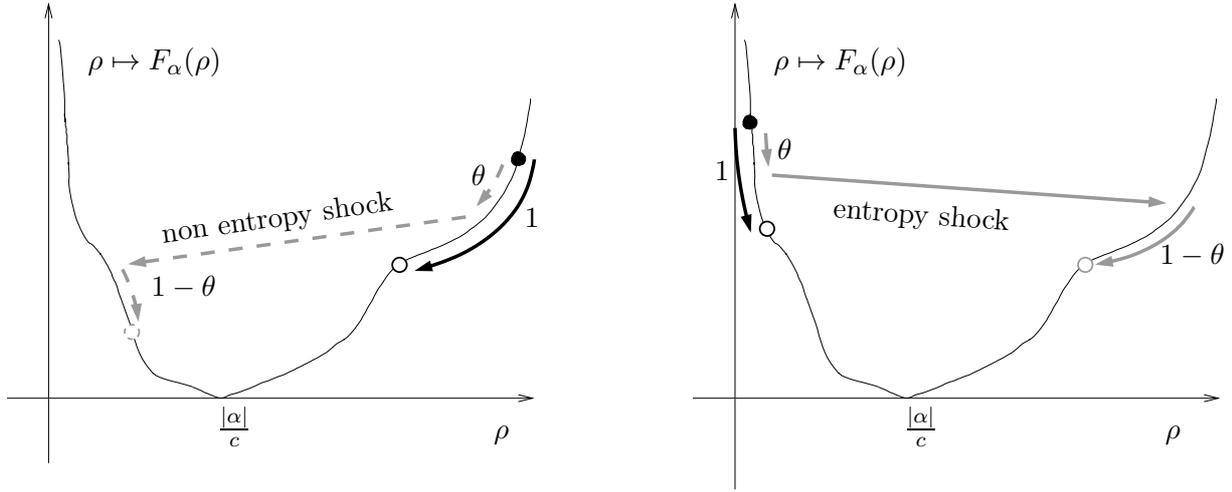}
 \caption{ How to reach the density $\rho_{+}$ (white circles) from the density $\rho_{-}$ (black circles) when $\alpha$ is positive. On the left is the supersonic case $c \rho_{-}<\alpha$; on the right is the subsonic case $\alpha \leq c \rho_{-}$. }
 \label{Fgerm}
 \end{figure}
\end{psfrags}
\end{proof}

In the following Corollary we extract properties of the germ $\mathcal{G}_{D}(v)$. 
\begin{cor} \label{sub->sub}
 If the pair $((\rho_{-},q_{-}),(\rho_{+},q_{+}))$ belongs to the germ $\mathcal{G}_{D}(v)$, then we necessary have:
\begin{itemize}
\item $ \alpha_{-} = \alpha_{+} := \alpha$;
\item if $\alpha>0$ and $c \rho_{-} \geq \alpha$,  then $ c \rho_{-} \geq c \rho_{+} \geq \alpha$;
\item if $\alpha<0$ and $c \rho_{+} \geq |\alpha|$,  then $ c \rho_{+} \geq c \rho_{-} \geq |\alpha|$;
\end{itemize}
\end{cor}
\begin{proof}
Suppose that $\alpha$ is positive. We already emphasized  in the proof of Theorem~\ref{Tdefgermgen} that if the velocity at the entry of the particle is subsonic (i.e. $c \rho_{-} \geq \alpha$) then there is no discontinuity in the solution and the solution remains subsonic: $c \rho_{-} \geq \alpha$. In that case, $F_{\alpha}(\rho_{-})-F_{\alpha}(\rho_{+}) =1$, and as $F_{\alpha}$ increases on $(\alpha/c, + \infty)$ we obtain that $\rho_{-} \geq \rho_{+}$.

\end{proof}
\begin{rem}
When the drag force depends only on $\alpha$, 
$$ F_{\alpha}(\rho)= \frac{1}{|D(\rho, \alpha)|} \left( \frac{\alpha^{2}}{\rho} + c^{2} \rho \right) $$
for some real $C$, and $F_{\alpha}$ has the remarkable property of being compatible with shocks at speed $v$ in the Euler equations:
$$ \forall \alpha \neq 0,\, \forall \rho >0, \, F_{\alpha}\left( \frac{\alpha^{2}}{c^{2}\rho} \right) = F_{\alpha}(\rho).$$  
It follows that a shock corresponds to a horizontal jump on the graph of $F_{\alpha}$, which is not the case in general. A consequence is that the right state $(\rho_{+}, q_{+})$ is the same whatever the value of $\theta$ is. The contrast between the two situations is depicted on Figure~\ref{Fgermgen} below.
\begin{psfrags} 
 \psfrag{al}{$ $}
  \psfrag{rho}{$\rho$}
  \psfrag{2ac}{$2 \alpha c$}
   \psfrag{ma}{$m_{\alpha}$}
  \psfrag{2ac+al}{$2 \alpha c+ \lambda \alpha$}
   \psfrag{al}{}
  \psfrag{a/c}{ $\frac{|\alpha_{-}|}{c} $}
  \psfrag{NEC}{ non entropic shock}
  \psfrag{rho ->f}{$\rho \mapsto F_{\alpha}(\rho)$}
  \psfrag{rho ->F}{$\rho \mapsto F_{\alpha}(\rho)$}
 \begin{figure}[H] 
 \centering
 \includegraphics[width=16cm]{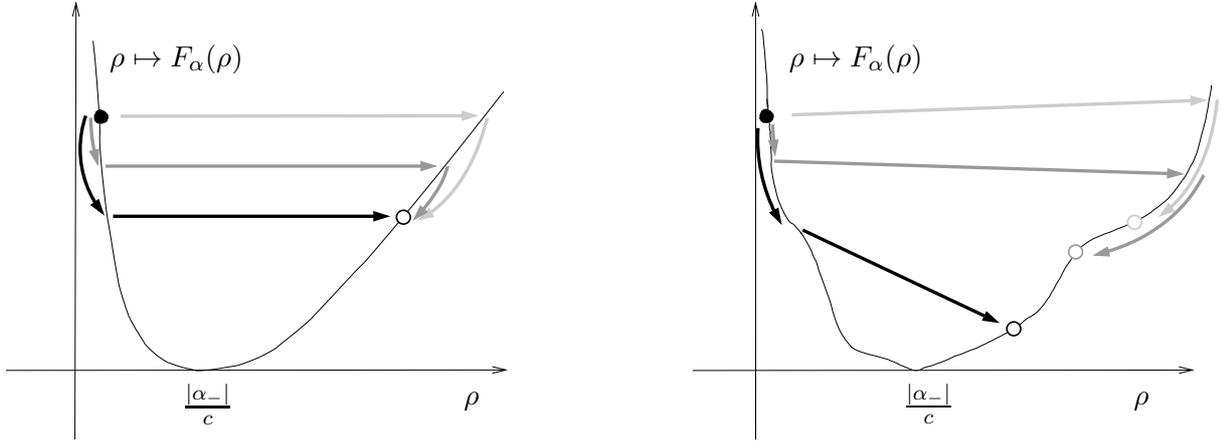}
 \caption{ Densities $\rho_{+}$ (white dots) accessible from $\rho_{-}$ with $\theta=0$ (light grey) $0<\theta<1$ (medium grey) and $\theta=1$ (black). On the left, the drag force depends only on $\alpha$, while on the right, it also depends on $\rho$.}
 \label{Fgermgen}
 \end{figure}
\end{psfrags}
In that case, the germ $\mathcal{G}_{D}(v)$ can be described more concisely: the second point of Definition~\ref{defgermgen} becomes
\begin{equation} \label{linearreformulation}
 \left( \frac{\alpha^{2}}{\rho_{-}} + c^{2} \rho_{-} \right) - \left( \frac{\alpha^{2}}{\rho_{+}} + c^{2} \rho_{+} \right) = \sign(\alpha) D(\rho, \alpha),
\end{equation}
and we still have the two inequalities of Corollary~\ref{sub->sub}
\end{rem}

\begin{rem} System~\eqref{eqregul} may not have any solution continuous on the whole interval $[-\eps/2, \eps/2]$ if $\rho_{-}$ is too close from $\alpha/c$ when $\alpha>0$. In the case of the linear drag force $D(\rho, \alpha)= \lambda \alpha$, where $\lambda \geq 0$ is a friction coefficient, an explicit computation shows that when $\alpha$ is positive there is no solution $\rho_{-}$ belongs to 
$$ \left[ \frac{\alpha}{c} + \frac{\lambda \alpha - \alpha \sqrt{4c \lambda + \lambda^{2}}}{2 c^{2}}, \frac{\alpha}{c} + \frac{\lambda \alpha + \alpha \sqrt{4c \lambda + \lambda^{2}}}{2 c^{2}} \right]. $$
\end{rem}

\subsection{Definition of the solution}
Let us now reformulate the ODE~\eqref{NewtonLaw}. The source term in~\eqref{ActionReaction} is the exact opposite of the left hand side in~\eqref{NewtonLaw}, so the total impulsion is formally conserved through time:
\begin{equation} \label{consq}
 \frac{ d}{dt} \left[ \int_{\R} \rho(t,x) u(t,x) dx + m h'(t) \right]=0.
\end{equation}
We can use this additional property of the model to give a precise definition of the ODE~\eqref{NewtonLaw}.
\begin{prop}
Let $(\rho,u)$ be a solution of~\eqref{MassConservation}-\eqref{ActionReaction} such that for almost every $t>0$, the traces around the particle exist and are such that
$$( (\rho_{-}(t), q_{-}(t)), (\rho_{+}(t), q_{+}(t)) )\in \mathcal{G}_{D}(h'(t)). $$
Then, it satisfies the conservation of total impulsion~\eqref{consq} if and only if for almost every $t>0$,
\begin{equation} \label{newODE}
 mh''(t) = c^{2}(\rho_{-}(t)-\rho_{+}(t)) \left(1 - \frac{(u_{-}(t)-h'(t))(u_{+}(t)-h'(t))}{c^{2}} \right),
\end{equation}
 where the subscripts $\pm$ indicates the left and right traces around the particle: $ \rho_{\pm}(t)= \rho(t,h(t)_{\pm})$.
 \end{prop}
 
\begin{proof} If $( (\rho_{-}(t), q_{-}(t)), (\rho_{+}(t), q_{+}(t)))$ belongs to $\mathcal{G}_{D}(h'(t))$, then $q_{-}(t) - h'(t) \rho_{-}(t)$ is equal to $q_{+}(t) - h'(t) \rho_{+}(t)$. As usual, we denote by $\alpha(t)$ this quantity, and replace $q= \rho u$ by $\alpha+ h' \rho$. We have
\begin{align*}
mh''(t)	&= -\partial_{t} \int_{\R} \rho(t,x) u(t,x) dx \\
		&= -\partial_{t} \int_{-\infty}^{h(t)} \rho(t,x) u(t,x) dx -\partial_{t} \int_{h(t)}^{+\infty} \rho(t,x) u(t,x) dx  \\
 		&= h'(\rho_{+} u_{+}-\rho_{-}u_{-}) + ( \rho_{-} u_{-}^{2} + c^{2} \rho_{-} ) - ( \rho_{+} u_{+}^{2} + c^{2} \rho_{+} ) \\
		&= h'^{2}( \rho_{+}  - \rho_{-})+  \rho_{-} \left(  \frac{\alpha^{2}}{\rho_{-}^{2}} + 2 h' \frac{\alpha} {\rho_{-}} + h'^{2}   \right) -  \rho_{+} \left(  \frac{\alpha^{2}}{\rho_{+}^{2}} + 2 h' \frac{\alpha}{\rho_{+}} + h'^{2}   \right) + c^{2}(\rho_{-}-\rho_{+}) \\
		&= h'^{2}( \rho_{+}  - \rho_{-}) + \alpha^{2}\left( \frac{1}{\rho_{-}}-\frac{1}{\rho_{+}} \right) + h'^{2}(\rho_{-}-\rho_{+}) + c^{2}(\rho_{-}-\rho_{+})\\
		&= (\rho_{-}-\rho_{+}) \left( c^{2} - \frac{\alpha^{2}}{\rho_{-}\rho_{+}}\right) \\
		&= c^{2}(\rho_{-}-\rho_{+}) \left(1 - \frac{(u_{-}-h')(u_{+}-h')}{c^{2}} \right).  \qedhere
\end{align*}

\end{proof}
\begin{rem}
The  ODE~\eqref{newODE} does not seem to depend on the drag force: this dependence is hidden in the assumption
$$( (\rho_{-}(t), q_{-}(t)), (\rho_{+}(t), q_{+}(t)) ) \in \mathcal{G}_{D}(h'(t)) $$
because the germ does depend on the drag force.
\end{rem}
\begin{rem}
 When the drag force depends only on $\alpha$, which is conserved through the particle, the initial ODE~\eqref{NewtonLaw} makes sense. In that case, it is not difficult to use the more concise description of the germ~\eqref{linearreformulation} to prove that~\eqref{NewtonLaw} and~\eqref{newODE} are equivalent.
\end{rem}

Thanks to the previous reformulation of the ODE and on the entropy inequality~\eqref{EntropyInequality}, we define the entropy solutions of the coupled problem~\eqref{eqcouple}:
\begin{defi}
Assume that $(\rho_{0},q_{0}) ∈ L^{\infty}(\R)^{2}$ and $v_{0} \in \R$. A triplet $(\rho, q, h) \in L^{\infty}(\R_{+} \times \R) \times L^{\infty}(\R_{+} \times \R) \times W^{2, \infty}_{loc}(\R_{+})$ is called an entropy solution of the problem~\eqref{eqcouple} if:
\begin{itemize}
\item $(\rho,q)$ is a weak solution of the isothermal Euler equations on the sets $ \{(t,x) \in \R^{*}_{+} \times \R: x > h(t) \}$ and $\{(t,x) \in \R^{*}_{+} \times \R: x < h(t) \}$.
\item For any entropy-entropy flux pair $(E,G)$ such that $E$ is convex and $\dis \partial_{q} E (\rho, q)= J \left( \frac{q}{\rho} \right)$, for any non-negative test function $\Phi \in \mathcal{C}_{0}^{\infty}(\R_{+}\times \R)$, we have
\begin{equation} \label{EI}
 \begin{aligned}
 \iint_{\R_{+} \times \R}  E(U) \partial_{t} \Phi \, dt \, dx + \iint_{\R_{+} \times \R}  G(U) \partial_{x} \Phi  \, dt \, dx  +  \int_{\R_{+}} m P(h'(t)) \partial_{t}(\Phi(t,h(t))) dt \\
 + \int_{\R} \Phi(0,\cdot) E(U_{0}) \, dx + m P(v_{0}) \Phi(0,v_{0}) \ \ \ \geq 0
\end{aligned}
\end{equation}

where $P$ is a given antiderivative of $J$;
\item For almost every $t>0$, the traces around the particle exist and belong to the germ at speed $h'(t)$:
$$( (\rho(t,h(t)_{-}),q(t,h(t)_{-})), (\rho(t,h(t)_{+}),q(t,h(t)_{+})) ) \in \mathcal{G}_{D}(h'(t)); $$
\item For almost every $t>0$, the particle is driven by the ODE:
$$ mh''(t) = c^{2}(\rho_{-}(t)-\rho_{+}(t)) \left(1 - \frac{(u_{-}(t)-h'(t))(u_{+}(t)-h'(t))}{c^{2}} \right). $$
\begin{rem}
 The entropy inequality~\eqref{EI} implies that the solution is an entropy solution of the Euler equations on the sets $\{x<h(t)\}$ and $\{x>h(t)\}$. Moreover, when the test function $\Phi$ tends to a Dirac measure at $(t, h(t))$ for which the traces exist,~\eqref{EI} yields that
 \begin{equation} \label{DissipEnergy}
 h'(t) (E(U_{-}) - E(U_{+})) - (G(U_{-}) - G(U_{+}))  + m h''(t) P'(h(t)) \leq 0
\end{equation}
 where we denote by $U_{\pm}$ the left and right traces around the particle: 
 $$U_{-}=(\rho(t, h(t)_{-}), q(t, h(t)_{-})).$$
In other words, the total energy is dissipated through the particle. This property is consistent with the Definition~\ref{defgermgen} of the germ and the ODE~\eqref{newODE}. Indeed, if we introduce $\alpha$ and replace $h''$ by its expression in~\eqref{newODE}, Equation~\eqref{DissipEnergy} becomes
$$\alpha \left[ \alpha^{2} \frac{\rho_{-}^{2}- \rho_{+}^{2}}{2 \rho_{-}^{2} \rho_{+}^{2}} + c^{2}( \ln(\rho_{+})- \ln(\rho_{-})) \right] \leq 0.$$
This holds true if $(U^{0}_{-}, U^{0}_{+})$ belongs to $\mathcal{G}_{D}(h'(t^{0}))$. This can be checked by treating sperately the subsonic case and the supersonic case. In the latter case, we have to use that $\rho_{+} \leq \frac{\alpha^{2}}{c^{2}\rho_{-}}$ which is easily deduced from Definition~\ref{defgermgen}.
\end{rem}

 \begin{rem}
 In the case of the scalar conservation law, it is not necessary to require the existence of the traces. Indeed, as the solution is a classic solution on $\{x<h' \}$, strong traces exist (see~\cite{P07} and~\cite{V01}). However, such a result is much harder to obtain in the system case.
\end{rem}
\end{itemize}
\end{defi}

\section{Riemann problem for a particle with a constant fixed velocity} \label{SlinearDF}
In this section we focus on the uncoupled problem where the particle has a constant speed equal to some given $v$ in $\R$. Moreover, we consider a class of very specific initial datum, which consists in piecewise constant functions for the density $\rho$ and  for the momentum $q=\rho u$, with a single discontinuity falling exactly on the initial position of the particle. The problem under study in this section is the Riemann problem:
\begin{equation} \label{pbR}
 \left\{
 \begin{array}{ll}
 \partial_t \rho + \partial_x q &= 0, \\
 \partial_t q + \partial_x \left( \frac{q^2}{\rho} +c^2 \rho \right) &= -D(\rho, \alpha) \delta_{v t}, \\
 \rho(0,x) &=\rho_{L} \mathbf{1}_{x<0}+\rho_{R} \mathbf{1}_{x>0}, \\
 q(0,x) &=q_{L} \mathbf{1}_{x<0}+q_{R} \mathbf{1}_{x>0},
\end{array} 
\right. 
\end{equation}
where $(\rho_{L},q_{L})$ and $(\rho_{R},q_{R})$ belong to $\R_{+}^{*} \times \R$. We recall once for all that $\alpha$ denotes the quantity $q-v\rho$. As the particle's trajectory is a straight line, we look for self-similar solutions of~\eqref{pbR}, i.e. solutions that only depend on $x/t$. This section is devoted to the proof of the following theorem.
\begin{thm} \label{allfriction2} Consider a  drag force
$$
\begin{array}{llll}
 D: &\R_{+}^{*} \times \R & \longrightarrow & \R \\
 &(\rho, \alpha) & \mapsto & D(\rho, \alpha)
\end{array}
$$
having the same sign as $\alpha$, vanishing in $\alpha=0$ and $\mathcal{C}^{1}$. Suppose that  $D$ is an increasing function of $\alpha$ and that $|D|$ is a decreasing function of $\rho$. Then  for all states $(\rho_{L},q_{L})$ and $(\rho_{R},q_{R})$ in $\R_{+}^{*}\times \R$ and for every particle velocity $v$ in $\R$, the Riemann problem~\eqref{pbR} has a unique self-similar solution.
\end{thm}
\begin{ex} All frictions of the form
$$\Gamma(\rho,u, h')= \rho^{n} |u-h'|^{m-1} (u-h')= \rho^{n-m} |\alpha|^{m-1} \alpha $$
with $n \geq 0$, $m \geq 1$ and $m \geq n$, verify the two conditions of Theorem~\ref{allfriction2}.
\end{ex}
Let $(\rho, q)$ be a self-similar solution of~\eqref{pbR} and denote by $(\rho_{-},q_{-})$ (respectively $(\rho_{+},q_{+})$)  the traces of the density and the momentum on the left (respectively on the right) of the particle, i.e. on the line $(t, (vt)_{-})$ (respectively on the line $(t, (vt)_{+})$). Then  
$$(\rho|_{L}, q |_{L}) = 
\begin{cases}
 (\rho, q) & \textrm{ on } x< vt \\
 (\rho_{-}, q_{-}) & \textrm{ on } x \geq vt
\end{cases}
\ \ \ \textrm{ and }
(\rho|_{R}, q |_{R}) =
\begin{cases}
 (\rho_{+}, q_{+}) & \textrm{ on } x \leq vt \\
 (\rho, q) & \textrm{ on } x > vt
\end{cases}
$$
are self-similar solution of the \emph{classical} isothermal Euler equations (without source term) for the initial datum
$$ (\rho(0,x), q(0,x))= ( \rho_{L} \mathbf{1}_{x<0}+\rho_{-} \mathbf{1}_{x>0}, q_{L} \mathbf{1}_{x<0}+q_{-} \mathbf{1}_{x>0}) $$
and
$$ (\rho(0,x), q(0,x))= ( \rho_{+} \mathbf{1}_{x<0}+\rho_{R} \mathbf{1}_{x>0}, q_{+} \mathbf{1}_{x<0}+q_{R} \mathbf{1}_{x>0} ).$$
Therefore, on $\{ (t,x), x<vt \}$, a solution of~\eqref{pbR}, is just the restriction on this set of a Riemann problem for the Euler equations. The same holds true on $\{ (t,x), x>vt \}$, for a different Riemann problem.
%

 In Section~\ref{acc states} we describe the set $U_{-}(\rho_{L},q_{L},v)$ of all the values that a Riemann solution for the classical Euler equations between $(\rho_{L},q_{L})$ and a state $(\underline{\rho},\underline{q}) \in \R_{+} \times \R$ can take on the line $x=vt$ and the set $U_{+}(\rho_{R},q_{R},v)$ of all the values that a Riemann solution  for the classical Euler equations between $(\bar{\rho},\bar{q}) \in \R_{+} \times \R$ and $(\rho_{R},q_{R})$ can take on the line $x=vt$. The traces $(\rho_{-},q_{-})$ and $(\rho_{+},q_{+})$ around the particle should be respectively chosen in those sets. The existence and uniqueness to the Riemann problem~\eqref{pbR} boils down to prove that there is a unique way to pick $(\rho_{-},q_{-})$ in $U_{-}(\rho_{L},q_{L},v)$ and $(\rho_{+},q_{+})$ in $U_{+}(\rho_{R},q_{R},v)$ such that
$((\rho_{-},q_{-}), (\rho_{+},q_{+}))$ belongs to the germ $\mathcal{G}_{D}(v)$. 
 In Section~\ref{pbR sub}, we prove that if $u_{L}-v \leq c$ and $u_{R}-v \geq -c$, any potential traces $(\rho_{\pm},q_{\pm})$ around the particle inherits the property $|u_{\pm}-v| \leq c$ and conclude in that case. It will be referred to as the subsonic case.  The other case, referred to as the supersonic case and studied in Section~\ref{pbR sup}, is more complicated because different types of solutions arise. 
\begin{rem}
 In all the sequel, the terms \emph{subsonic} and \emph{supersonic} are used \emph{in the framework of the particle}. For example we say that a solution is subsonic if the difference between the velocity of the particle and the fluid's velocity on both side of the particle is smaller than the speed of sound $c$.
\end{rem}

\subsection{Accessible states around the particle} \label{acc states}
Let us start with some very classical results on the isothermal Euler equations without source term that will be useful to determine the solution of~\eqref{pbR}.
\begin{lemma} \label{Euler1}
The isothermal Euler equations without source term
$$
\left\{
 \begin{array}{ll}
 \partial_t \rho + \partial_x q &= 0, \\
 \partial_t q + \partial_x (\frac{q^2}{\rho} +c^2 \rho) &= 0,
 \end{array}
\right. 
$$
is a strictly hyperbolic system. The eigenvalues of its Jacobian matrix are 
$$ \lambda_i( \rho, q)= \frac{q}{\rho}+ (-1)^i c $$
and the corresponding normalized eigenvectors are
$$ r_i( \rho, q)= \left( \begin{array}{c}
                          \frac{(-1)^i \rho}{c} \\
			  \frac{(-1)^i q}{c} + \rho
                         \end{array}
                           \right) . 
$$
They define two genuinely nonlinear fields. 
The $i$-th rarefaction waves express
$$ \rho(s)= \left\{
\begin{array}{ll}
  \rho_L 								& \mbox{ if } s \leq s_L= \frac{q_L}{\rho_L} + (-1)^i c, \\
  \rho_L e^{\frac{(-1)^i}{c}(s-s_L)} 			& \mbox{ if }  s_L \leq s \leq s_R, \\
  \rho_R= \rho_L e^{\frac{(-1)^i}{c}(s_R-s_L)} 	& \mbox{ if }  s_R \leq s, 
\end{array}
\right. $$
and
$$ q(s)= \left\{
\begin{array}{ll}
 q_L  													& \mbox{ if } s \leq s_L= \frac{q_L}{\rho_L} + (-1)^i c, \\
 \left[ q_L + \rho_L (s-s_L) \right] e^{\frac{(-1)^i}{c}(s-s_L)}  			& \mbox{ if }  s_L \leq s \leq s_R, \\
 q_R= \left[ q_L + \rho_L (s_{R}-s_L) \right] e^{\frac{(-1)^i}{c}(s_R-s_L)} 	& \mbox{ if }  s_R \leq s .
\end{array}
\right. $$
The speed of the $i$-shock is
$$ \sigma_{i}= \frac{q_L}{\rho_L}+ (-1)^i c \sqrt{ \frac{ \rho_R}{\rho_L} } $$
and shocks are entropy satisfying if and only if $u_{L}\geq u_{R}$.
\end{lemma}
We recall below a well-known result on the structure of the Riemann solution.
\begin{prop} For all $(\rho_{a},q_{a})$ and $(\rho_{b},q_{b})$ in $\R_{+}^{*} \times \R$,  there exists a unique self-similar entropy solution of the Riemann problem
 \begin{equation*} 
 \left\{
 \begin{array}{ll}
 \partial_t \rho + \partial_x q &= 0, \\
 \partial_t q + \partial_x \left( \frac{q^2}{\rho} +c^2 \rho \right) &=0, \\
 \rho(0,x) &=\rho_{a} \mathbf{1}_{x<0}+\rho_{b} \mathbf{1}_{x>0}, \\
 q(0,x) &=q_{a} \mathbf{1}_{x<0}+q_{b} \mathbf{1}_{x>0}. \\
\end{array} 
\right. 
\end{equation*}
It consists of the succession of a $1$-wave (rarefaction or shock) linking $(\rho_{L}, q_{L})$ to an intermediate state $(\rho_{*},q_{*})$ followed by a $2$-wave (rarefaction or shock) linking $(\rho_{*}, q_{*})$ to $(\rho_{R},q_{R})$. We denote by $$W(x/t; (\rho_{a},q_{a}), (\rho_{b},q_{b}))$$
this unique solution.
\end{prop}
\begin{proof}
 The proof of those two results are classical and can be found for example in~\cite{GR96}.
\end{proof}
Those tools allow us to describe the set of the accessible states from $(\rho_{L},q_{L})$ on the left of a particle moving at speed $v$
 $$ \dis U_{-}(\rho_{L},q_{L},v) = \left\{ W(v_{-}; (\rho_{L},q_{L}), (\underline{\rho},\underline{q})),  \  (\underline{\rho},\underline{q})  \in \R_{+}^{*} \times \R \right\}. $$
 and the set of the accessible states  from $(\rho_{R},q_{R})$ on the right of a particle moving at speed $v$ 
 $$ \dis U_{+}(\rho_{R},q_{R},v) = \left\{ W(v_{+}; (\bar{\rho},\bar{q}), (\rho_{R},q_{R})),  \  (\bar{\rho},\bar{q})  \in \R_{+}^{*} \times \R \right\} $$
 in which the left and right traces around the particle must be chosen. According to Definition~\ref{defgermgen}, the quantity $q-v \rho$ must be conserved through the particle, so it is more convenient to reason with the variables $(\rho, q-v\rho)$ rather than with the initial unknowns $(\rho, q)$. We introduce the sets
 \begin{equation} \label{V-}
 \dis V_{-}(\rho_{L},\alpha_{L},v) = \left\{ (\rho_{-}, q_{-}-v \rho_{-}) \, : \,  (\rho_{-},q_{-}) \in  \dis U_{-}(\rho_{L},\alpha_{L}+ v \rho_{L},v)  \right\} 
\end{equation}
and
\begin{equation} \label{V+}
  \dis V_{+}(\rho_{R},\alpha_{R},v) =  \left\{ (\rho_{+}, q_{+}-v \rho_{+}) \, : \,  (\rho_{+},q_{+}) \in  \dis U_{+}(\rho_{R},\alpha_{R}+ v \rho_{R},v)  \right\}.
\end{equation}
Proving that there exists a unique solution to the Riemann problem~\eqref{pbR} is equivalent to prove that
$$ \dis \mathcal{G}_{D}(v) \cap  \left( U_{-}(\rho_{L},q_{L},v) \times U_{+}(\rho_{R},q_{R},v) \right)$$
consists in a unique pair of states $\left( (\rho_{-},q_{-}), (\rho_{+},q_{+}) \right)$. We now give a precise description of the sets $\dis V_{-}(\rho_{L},\alpha_{L},v)$ and $ \dis  V_{+}(\rho_{R},\alpha_{R},v)$.

\begin{lemma} \label{U-}
 Let $(\rho_{L},\alpha_{L}) \in \R_{+}^{*} \times \R$  and $v \in \R$.   Then
 $$ \dis V_{-}(\rho_{L},\alpha_{L},v) = \{ (\rho_{L}, \alpha_{L} ) \} \cup \Gamma_{-}^{sub} \cup \Omega_{-}^{sup}, $$
 where  $\Gamma_{-}^{sub}$ is the graph of a decreasing function $f_{-}^{sub}: [\rho_{L, ex}, +\infty ] \rightarrow \R$ for some $\rho_{L,ex}>0$, included in  $\{ (\rho, \alpha): -c \rho \leq \alpha  \leq c \rho  \}$, and $\Omega_{-}^{sup}$ is the strict hypograph  of a decreasing function $f_{-}^{sup}: (0, +\infty ) \rightarrow \R$, included in  $\{ (\rho, \alpha): \alpha < -c \rho  \}$.
\end{lemma}
\begin{lemma} \label{U+}
 Let $(\rho_{R},\alpha_{R}) \in \R_{+}^{*} \times \R$  and $v \in \R$. Then
 $$ \dis V_{+}(\rho_{R},\alpha_{R},v) = \{ (\rho_{R}, \alpha_{R} ) \} \cup \Gamma_{+}^{sub} \cup \Omega_{+}^{sup} $$
 where  $\Gamma_{+}^{sub}$ is the graph of a increasing function $f_{+}^{sub}: [\rho_{R, ex}, +\infty ] \rightarrow \R$ for some $\rho_{R,ex}>0$, included in  $\{ (\rho, \alpha): -c \rho \leq \alpha  \leq c \rho  \}$, and $\Omega_{+}^{sup}$ is the strict epigraph  of a increasing function $f_{+}^{sup}: (0, +\infty ) \rightarrow \R$, included in  $\{ (\rho, \alpha): \alpha > c \rho  \}$.
\end{lemma}
Those sets are respectively depicted on the left and on the right of Figure~\ref{FAcc}. The subscript $ex$ refers to the extremity of $\Gamma_{-}^{sub}$ and $\Gamma_{+}^{sub}$.
\begin{psfrags}
\psfrag{a}{$\alpha$}
\psfrag{r}{$\rho$}
\psfrag{G-}{$\Gamma_{-}^{sub}$}
\psfrag{O-}{$\Omega_{-}^{sup}$} 
\psfrag{G+}{$\Gamma_{+}^{sub}$}
\psfrag{O+}{$\Omega_{+}^{sup}$} 
\psfrag{rL,qL}{$(\rho_{L},\alpha_{L})$}
\psfrag{rlex}{$\rho_{L,ex}$}
\psfrag{rrex}{$\rho_{R,ex}$}
\psfrag{rR,qR}{$(\rho_{R},\alpha_{R})$}
\psfrag{a=cr}[][][1][26.5]{$\alpha=c \rho$}
\psfrag{a=-cr}[][][1][-23.5]{$\alpha=-c \rho$}
\begin{figure}[H]
\centering
\includegraphics[width=15cm]{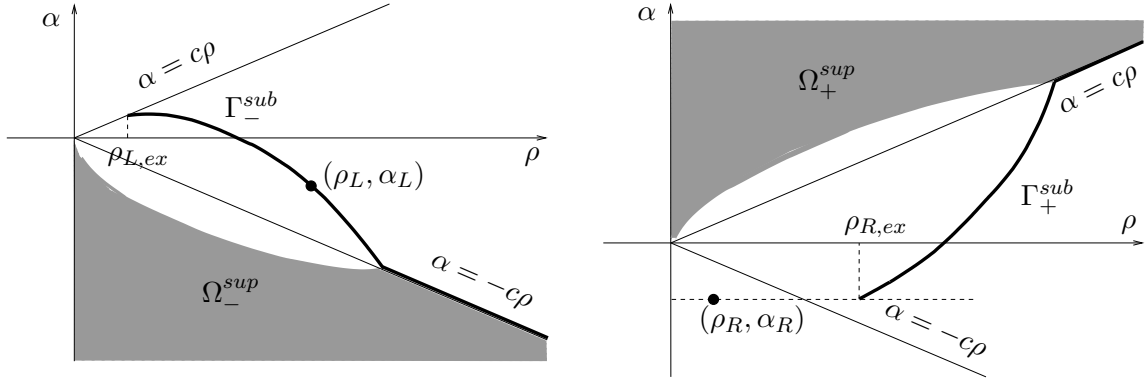}
\caption{ On the left, the set  $\dis V_{-}(\rho_{L},\alpha_{L},v)$ when $-c \rho_{L}< \alpha_{L}< c \rho_{L}$. On the right, the set  $V_{+}(\rho_{R},\alpha_{R},v)$ when $\alpha_{R}< c \rho_{R}$.}
 \label{FAcc}
 \end{figure}
\end{psfrags}

\begin{rem} By definition of $\alpha$ we have
 $$ \{ (\rho, \alpha): -c \rho \leq \alpha  \leq c \rho  \} = \{ (\rho, u) -c \leq u-v \leq c \}, $$
 so $\Gamma_{-}^{sub}$ and $\Gamma_{+}^{sub}$ contain only subsonic states (in the framework of the particle), while $\Omega_{-}^{sup}$ and $\Omega_{+}^{sup}$ contain only supersonic states (in the framework of the particle).
\end{rem}

\begin{proof}[Proof (Lemma~\ref{U-})] We fix the particle velocity $v$ and a left state $(\rho_{L},\alpha_{L})$. The element of $(\rho, \alpha)$ of the set $V_{-}(\rho_{L},\alpha_{L},v)$ is the value on the line $x=vt$ of a Riemann problem with $(\rho_{L}, \alpha_{L})$ on its left. Thus it is linked with $(\rho_{L}, \alpha_{L})$ by a succession of two waves, both traveling slower than $v$. Let us first exhibit all the states $(\rho_{*}, q_{*})$ that can be linked to $(\rho_{L},q_{L})$ with a $1$-wave traveling at speed smaller than $v$. According to Lemma~\eqref{Euler1}, with a $1$-rarefaction wave we can reach all the states $(\rho, q)$ in the set
$$ \dis \left\{ \left(  \rho_L e^{-\frac{(s-(u_{L}-c))}{c}} , \left[ q_L + \rho_L (s-(u_L-c)) \right] e^{-\frac{(s-(u_{L}-c))}{c}} \right) ,\ u_{L}-c \leq s \leq v  \right\} . $$
This set is empty if $u_{L}-c>v$. Parametrized by $\rho$ it rewrites
$$  \dis \left\{ \left(  \rho , \left[ u_{L}-c \ln \left(   \frac{\rho}{\rho_{L}} \right)  \right] \rho \right) , \rho_{L} e^{- \frac{v-(u_{L}-c)}{c}} \leq \rho \leq \rho_{L} \right\} . $$
Let us denote in this case $\dis \rho_{L,ex}= \rho_{L} e^{- \frac{v-(u_{L}-c)}{c}}$. The states $(\rho_{*},q_{*})$ accessible through a $1$-shock traveling slower than $v$ satisfy
$$
\begin{cases}
  \frac{q_{L}}{\rho_{L}}-c \sqrt{\frac{\rho_{*}}{\rho_{L}}} < v  \ \ \ \textrm{and} \ \ \  \rho_{*}>\rho_{L}, \\
  q_{*}=q_{L}+ \left(   \frac{q_{L}}{\rho_{L}}-c \sqrt{\frac{\rho_{*}}{\rho_{L}}} \right) (\rho_{*}-\rho_{L}).
\end{cases}
$$
We easily pass to the $(\rho, \alpha)$ variable: with a $1$-wave traveling slower than $v$, we can reach all the states $(\rho, \alpha)$ with $\alpha= f_{-}(\rho)$, where
$$
f_{-}(\rho)= 
\begin{cases}
\left[ \frac{\alpha_{L}}{\rho_{L}}-c \ln \left( \frac{\rho}{\rho_{L}} \right) \right] \rho & \textrm{ if } \rho_{L,ex} \leq \rho \leq \rho_{L}, \\
\alpha_{L} + \left( \frac{\alpha_{L}}{\rho_{L}}-c \sqrt{\frac{\rho}{\rho_{L}}}\right)(\rho-\rho_{L})  & \textrm{ if } \max( \rho_{L,ex},\rho_{L})< \rho,
\end{cases}
$$
and
$$
\rho_{L,ex}= 
\begin{cases}
 \rho_{L} e^{- \frac{v-s_{L}}{c}} & \textrm{ if } s_{L}=u_{L}-c \leq v, \\ 
 \left( \frac{u_{L}-v}{c} \right)^{2} \rho_{L} & \textrm{ if } u_{L}-c > v.
\end{cases}
$$
 If $u_{L}-v<c$,  this graph regroups all the $1$-shocks and the $1$-rarefaction waves leading to a density higher than $\rho_{L,ex}=\rho_{L} e^{- \frac{v-(u_{L}-c)}{c}}$, while if
$u_{L}-v \geq c$, this graphs contains only the $1$-shocks leading to a density higher than $\dis \rho_{L,ex}= \rho_{L}  \left( \frac{u_{L}-v}{c} \right)^{2} = \frac{\alpha_{L}}{c^{2} \rho_{L}}$ . We check that $f_{-}$ is concave and decreasing. Moreover, if  $u_{L}-c \leq v$, $f_{-}(\rho_{L,ex})=c \rho_{L,ex}$ and $f_{-}'(\rho_{L,ex})=0$, while if $u_{L}-c > v$, $f_{-}(\rho_{L,ex})= \alpha_{L}$. In particular, $f_{-}(\rho) <c \rho$ for any $\rho>\rho_{L,ex}$ as shown on Figure~\eqref{A1W}.
\begin{psfrags} 
 \psfrag{a}{$\alpha$}
 \psfrag{r}{$\rho$}
 \psfrag{rlex}{$\rho_{L,ex}$}
  \psfrag{r,q}{ }
 \psfrag{a=cr}[][][1][26.5]{$\alpha=c \rho$}
 \psfrag{a=-cr}[][][1][-23.5]{$\alpha=-c \rho$}
 \begin{figure}[H]
 \centering
 \includegraphics[width=13cm]{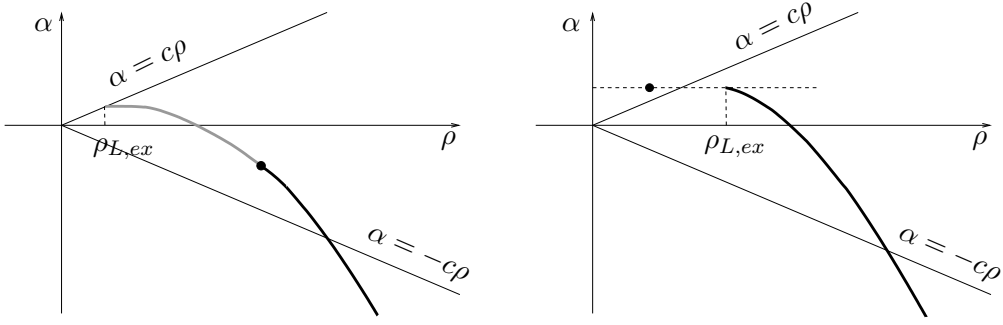}
 \caption{Accessible states via a $1$-wave starting from $(\rho_{L},q_{L})$(black dot) in the $(\rho, \alpha)$-plane. Left, the case $u_{L}-v\leq c$ and right, the case $u_{L}-v > c$. In gray are the $1$-rarefaction waves and in black are the $1$-shocks.  }
 \label{A1W}
 \end{figure}
\end{psfrags}
Let us now stop at any state $(\rho_{*},q_{*})$ belonging to the graph of $f_{-}$, and continue with a $2$-wave traveling at speed smaller than $v$.
The set of all the states $(\rho_{-},q_{-})$ that can be joined from $(\rho_{*},q_{*})$ with a $2$-wave traveling at speed smaller than $v$ is
$$ \dis \left\{ \left(  \rho_* e^{\frac{(s-(u_{*}+c))}{c}} , \left[ q_* + \rho_L (s-s_*) \right] e^{\frac{(s-(u_{*}+c))}{c}} \right) ,\ u_{*}+c \leq s \leq v  \right\}. $$
This set is empty if $u_{*}+c>v$, and can be parametrized by $\rho$ by
$$  \dis \left\{ \left(  \rho , \left[ u_{*}+c \ln \left(   \frac{\rho}{\rho_{*}} \right)  \right] \rho \right) , \rho_{*}  \leq \rho \leq \rho_{*} e^{ \frac{v-(u_{*}+c)}{c}} \right\} . $$
With a $2$-shock slower than $v$, we can reach all the states $(\rho_{-},q_{-})$ such that
$$
\begin{cases}
  \frac{q_{*}}{\rho_{*}}+c \sqrt{\frac{\rho_{-}}{\rho_{*}}} < v  \ \ \ \textrm{and} \ \ \  \rho_{-}<\rho_{*}, \\
  q_{-}=q_{*}+ \left(   \frac{q_{*}}{\rho_{*}}+c \sqrt{\frac{\rho_{-}}{\rho_{*}}} \right) (\rho_{-}-\rho_{*}).
\end{cases}
$$
Therefore the $2$-waves traveling at speed smaller than $v$, starting from $(\rho_{*},q_{*})$, are:
\begin{itemize}
 \item If $u_{*}-v \leq -c$, all the $2$-shocks and the $2$-rarefaction waves leading to a density smaller than $\rho_{*} e^{ \frac{v-(u_{*}+c)}{c}}$ ;
 \item If $-c < u_{*}-v <0$, only the $2$-shock leading to a density smaller than 
 $$\dis \rho_{*,ex}= \rho_{*} \left( \frac{v-u_{*}}{c} \right)^{2} = \frac{\alpha_{*}^{2}}{c^{2} \rho_{*}};$$
  \item If $u_{*}-v \geq 0$, there are no such $2$-waves.
\end{itemize}
The Figure~\eqref{AW2} resumes the first two cases.
 \begin{psfrags}
 \psfrag{a}{$\alpha$}
 \psfrag{r}{$\rho$}
 \psfrag{r*}{$\rho_{*}$}
 \psfrag{r*ex}{$\rho_{*,ex}$}
 \psfrag{r*ex=r*}{$\rho_{*,ex}=\rho_{*}$}
 \psfrag{r,q}{$(\rho_{L},\alpha_{L})$}
 \psfrag{r*,q*}{$(\rho_{*},\alpha_{*})$}
 \psfrag{a=cr}[][][1][26.5]{$\alpha=c \rho$}
 \psfrag{a=-cr}[][][1][-25.5]{$\alpha=-c \rho$}
 \begin{figure}[H]
 \centering
 \includegraphics[width=13cm]{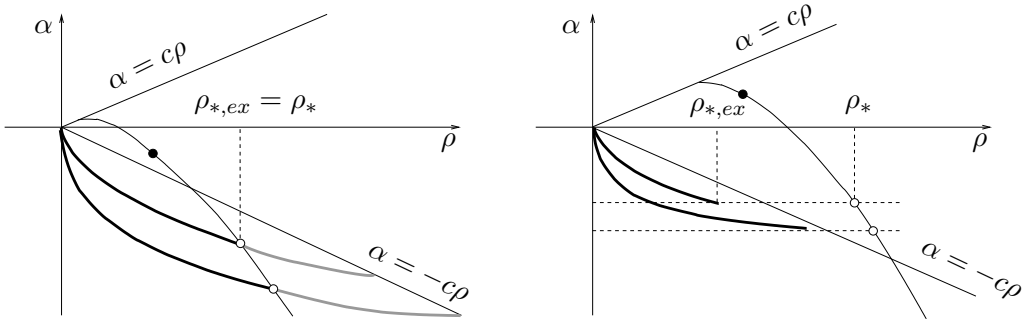}
 \caption{Accessible states from $(\rho_{L},q_{L})$ (black dot) via a $1$-wave stopped in $(\rho_{*},q_{*})$ (white dots) followed by a $2$-wave in the $(\rho, \alpha)$-plane. Left, the case where $u_{*}-v\leq -c$ and right, the case $-c <u_{L*}-v <0$. In gray are the $2$-rarefaction waves and in black are the $2$-shocks.  }
 \label{AW2}
 \end{figure}
\end{psfrags}
We now prove that, as emphasized in Figure~\ref{AW2}, all the states reached from $(\rho_{*},q_{*})$ through  a $2$-wave traveling at speed smaller than $v$ are such that $\alpha \leq -c \rho$. This is easy to check for the $2$-rarefaction waves and for the $2$-shock when $u_{*}-v \leq -c$.  When $-c <u_{*}-v < 0$, the reachable densities are smaller than $\frac{\alpha_{*}^{2}}{c^{2} \rho_{*}}< \rho_{*}$. We want to prove that
$$\alpha=\rho \left( \frac{\alpha_{*}}{\rho_{*}}+ c \sqrt{\frac{\rho}{\rho_{*}}}- c \sqrt{\frac{\rho_{*}}{\rho}} \right) \leq -c \rho. $$
This is true because the function $\rho \mapsto \frac{\alpha_{*}}{\rho_{*}}+ c \sqrt{\frac{\rho}{\rho_{*}}}- c \sqrt{\frac{\rho_{*}}{\rho}}$ is increasing, and is equal to $\frac{-c^{2} \rho_{*}}{|\alpha_{*}|}$, which is smaller than $-c$.

To conclude, the set of the states reached from $(\rho_{L},\alpha_{L})$ through a $1$-wave slower than the particle is the graph of $f_{-}$. The set of the stated reached from $(\rho_{L},\alpha_{L})$ through a $1$-wave followed by a $2$-wave, both of them traveling at speed smaller than $v$, is a family of curves, entirely included in $\{(\rho,\alpha), \alpha \leq -c \rho \}$. Those curves are portions of Lax curves, which fill the half plane $\R_{+}^{*} \times \R$ and do not cross each other. For a fixes $\alpha<0$ such that $f_{-}(\rho) \geq -c \rho$, it is not possible to reach densities higher than the density obtain thanks to a shock at speed $v$, which is $\frac{f_{-}(\rho)^{2}}{c^{2} \rho}$.
We can separate those states in two categories, as depicted on the left of Figure~\ref{FAcc}:
\begin{itemize}
 \item the subsonic ones (i.e the ones such that $-c \rho \leq \alpha \leq c \rho$), which constitute the graph of the function 
$$ f_{-}^{sub}(\rho)=
\begin{cases}
 f_{-}(\rho) & \textrm{ if } \rho_{L,ex} \rho \leq \rho_{-c} \\
 -c \rho & \textrm{ if } \rho_{-c} \leq \rho  
\end{cases}
$$
where $\rho_{-c}$ is the only density such that $f_{-}(\rho_{-c})=-c \rho$. We regroup those states in $\Gamma_{-}^{sub}$;
\item the supersonic ones, and more precisely the states such that $\alpha<-c \rho$, which formed the hypograph of the function
$$ f_{-}^{sup}(\rho)=
\begin{cases}
 f_{-} \left( \frac{f_{-}(\rho)^{2}}{c^{2} \rho}\right) & \textrm{ if } 0<\rho<\rho_{-c} \\
 -c \rho &  \textrm{ if } \rho_{-c} \leq \rho.
\end{cases}
$$
We regroup those states in $\Omega_{-}^{sup}$. \qedhere
\end{itemize}
\end{proof}

\begin{proof}[Proof (Lemma~\ref{U+})]
 We do not prove this lemma which is exactly similar to Lemma~\ref{U-}. The curve of the states accessible by a $2$-wave traveling faster than $v$ and ending in $(\rho_{R},q_{R})$, can be parametrized by
 $$
f_{+}(\rho)= 
\begin{cases}
\left[ \frac{\alpha_{R}}{\rho_{R}}+c \ln \left( \frac{\rho}{\rho_{R}} \right) \right] \rho & \textrm{ if } \rho_{R,ex} \leq \rho \leq \rho_{R}, \\
\alpha_{R} + \left( \frac{\alpha_{R}}{\rho_{R}}+c \sqrt{\frac{\rho}{\rho_{R}}}\right)(\rho-\rho_{R})  & \textrm{ if } \max( \rho_{R,ex},\rho_{R})< \rho,
\end{cases}
$$
where
$$
\rho_{R,ex}= 
\begin{cases}
 \rho_{R} e^{ \frac{v-(u_{R}+c)}{c}} & \textrm{ if } u_{R}+c \geq v, \\ 
 \left( \frac{v-u_{R}}{c} \right)^{2} \rho_{R} & \textrm{ if } u_{R}+c < v.
\end{cases}
$$
If $u_{R}-v \geq -c$ it is possible to follow all the $2$-shocks and some $2$-rarefaction waves. In this case $f_{+}(\rho_{R,ex})$ belongs to the line $\alpha=-c \rho$ and $f_{+}'(\rho_{R,ex})=0$. On the contrary if $u_{R}-v < -c$, we can only follow $2$-shocks, and the state $(\rho_{R,ex},\alpha_{R,ex})$ with the highest density we can reach verifies
$$ \alpha_{R,ex}= \alpha_{R} \ \ \ \textrm{ and } \ \ \ \rho_{R,ex}= \frac{\alpha_{R}^{2}}{c^{2} \rho_{R}}. $$
We easily check that that $f_{+}$ is convex, increasing and crosses the line $\alpha=c \rho$ for a unique density that we denote by $\rho_{c}$.

After similar computations for the $1$-wave ending on a state $(\rho_{*},q_{*}=f_{+}(\rho_{*}))$ we obtain
$$ f_{+}^{sub}(\rho)=
\begin{cases}
 f_{+}(\rho) & \textrm{ if } \rho_{R,ex}  \leq \rho \leq \rho_{c} ,\\
 c \rho & \textrm{ if } \rho_{c} \leq \rho,  
\end{cases}
$$
and
$$
f_{+}^{sup}=
\begin{cases}
 f_{+} \left( \frac{f_{+}(\rho)^{2}}{c^{2} \rho}\right) & \textrm{ if } 0<\rho<\rho_{c}, \\
 c \rho &  \textrm{ if } \rho_{c} \leq \rho. 
\end{cases} 
$$\qedhere
\end{proof}

\subsection{Resolution of the Riemann problem in the subsonic case} \label{pbR sub}
We are now in position to solve the Riemann problem~\eqref{pbR} for a particle moving at speed $v$. In the sequel, we denote by $(\rho_{+},q_{+})$ the trace on the left of the particle, i.e. on the line $x=(vt)_{-}$ and by $(\rho_{-},q_{-})$ the trace on its right, i.e. on the line $x=(vt)_{+}$. In this section, we treat a special case implying that $(\rho_{L},\alpha_{L})$ does not belong to $\Omega_{+}^{sup}$ and $(\rho_{R},\alpha_{R})$ does not belong to $\Omega_{-}^{sup}$. A consequence is that the left trace belongs to $\Gamma_{-}^{sub}$ and the right trace belong to $\Gamma_{+}^{sub}$.
\begin{lemma} \label{L:sub->sub}
 If $(\rho_{L},q_{L})$ and $(\rho_{R},q_{R})$ verify
  $$ u_{L}-v \leq c \ \ \ \textrm{ and } \ \ \ -c \leq u_{R}-v  $$
 then $(\rho_{-},q_{-})$ and $(\rho_{+},q_{+})$ are necessary subsonic, i.e,
  $$ -c \leq u_{-}-v \leq c \ \ \ \textrm{ and } \ \ \ -c \leq u_{+}-v \leq c. $$
\end{lemma}
\begin{proof}
This is a straightforward consequence of the form of $V_{-}(\rho_{L},\alpha_{L},v)$ and $V_{+}(\rho_{R},\alpha_{R},v)$ exhibited in Lemmas~\ref{U-} and~\ref{U+} and resumed on Figure~\ref{FAcc}. Definition~\ref{defgermgen} implies that $q_{-}-v \rho_{-}= q_{+}-v \rho_{+}$. We denote by $\alpha$ this quantity. If $\alpha=0$ there is nothing to prove. Suppose $\alpha>0$. Lemma~\ref{U-} shows that $V_{-}(\rho_{L},\alpha_{L},v)$ contains only states such that $u-v\leq c$. As a consequence, we necessary have that $c \rho_{-}\geq \alpha$. The velocity at the entry of the particle is subsonic, thus it is also subsonic at its exit:  $\rho_{+}\geq \frac{\alpha}{c}$ (see Corollary~\ref{sub->sub}).
  If $\alpha<0$, Lemma~\ref{U-} shows that $V_{+}(\rho_{R},\alpha_{R},v)$ contains only states verifying $u-v \geq -c$. Using Corollary~\ref{sub->sub} again, we obtain that both $\rho_{+}$ and $\rho_{-}$ are larger than $\frac{|\alpha|}{c}$.
\end{proof}
We are now looking for $(\rho_{-},\alpha_{-})$ in $\Gamma_{-}^{sub}$ and $(\rho_{+},\alpha_{+})$ in $\Gamma_{+}^{sub}$. The quantity $\alpha=q-v\rho$ being conserved through the particle, it is more convenient to parametrize the accessible sets $\Gamma_{-}^{sub}$ and $\Gamma_{+}^{sub}$  by $\alpha$ rather than by $\rho$. For this purpose, we introduce $g_{-}^{sub}$ and $g_{+}^{sub}$, the inverses of $f_{-}^{sub}$ and $f_{+}^{sub}$. They exist because these two functions are strictly monotone. We now have
$$ \Gamma_{-}^{sub}=\{ (\rho, \alpha), \rho=g_{-}^{sub}(\alpha), \ \alpha \leq f_{-}^{sub}(\rho_{L,ex}) \} $$ 
and
$$ \Gamma_{+}^{sub}=\{ (\rho, \alpha), \rho=g_{+}^{sub}(\alpha), \ \alpha \geq f_{+}^{sub}(\rho_{R,ex}) \} .$$
\begin{lemma} \label{triangle}
If $D$ has the same sign as $\alpha$ an is a nondecreasing function of $\alpha$, the function
  $$ \Delta(\alpha)= F_{\alpha}(g_{-}^{sub}(\alpha))- F_{\alpha}(g_{+}^{sub}(\alpha))$$
is strictly decreasing on any interval included in $(0, + \infty)$ where $g_{-}^{sub} \geq g_{+}^{sub}$ and on any interval included in $(-\infty, 0)$ where $g_{-}^{sub} \leq g_{+}^{sub}$;
\end{lemma}

\begin{proof}  We compute the derivative of $\Delta$:
 $$ 
\begin{aligned}
 \left[ F_{\alpha} (g_{-}^{sub}(\alpha)) - F_{\alpha}(g_{+}^{sub}(\alpha))  \right]' 
 	= & F_{ \alpha}'(g_{-}^{sub}(\alpha))(g_{-}^{sub})'(\alpha) + \left[ \frac{\partial}{\partial \alpha}F_{\alpha}  \right] (g_{-}^{sub}(\alpha))  \\
	& \ \ \ - F_{ \alpha}'(g_{+}^{sub}(\alpha))(g_{+}^{sub})'(\alpha)  - \left[ \frac{\partial}{\partial \alpha}F_{\alpha}  \right] (g_{+}^{sub}(\alpha)) 
\end{aligned}
 $$
On the one hand, $\alpha \mapsto g_{-}^{sub}(\alpha)$ decreases,  $\alpha \mapsto g_{+}^{sub}(\alpha)$ increases. We recall that for all $\alpha$ in there interval of definition, $c g_{\pm}^{sub}(\alpha) \geq |\alpha| $. If  $ c \rho$ is larger than $|\alpha|$, 
 $$ 
 F_{ \alpha}'(\rho)= \frac{1}{|D(\rho, \alpha)|} \left( c^{2}- \frac{\alpha^{2}}{\rho^{2}}\right) \geq 0 .
 $$
 and it follows that
  $$  F_{ \alpha}'(g_{-}^{sub}(\alpha))(g_{-}^{sub})'(\alpha) - F_{ \alpha}'(g_{+}^{sub}(\alpha))(g_{+}^{sub})'(\alpha) \leq 0.
 $$
On the other hand, by definition of $F_{\alpha}$, we have
 $$ 
 \begin{aligned}
 \left[ \frac{\partial}{\partial \alpha}F_{\alpha}  \right] (\rho_{2}) -  \left[ \frac{\partial}{\partial \alpha}F_{\alpha}  \right] (\rho_{2}) 
 &= \int_{ \rho_{1}}^{\rho_{2}}  \frac{\partial}{\partial \alpha} \left[ \frac{1}{|D(r,\alpha)|} \left( c^{2}- \frac{\alpha^{2}}{r^{2}}\right) \right] dr \\
 &= \int_{ \rho_{1}}^{\rho_{2}} \left[ \frac{- \sign(D(r, \alpha)) \partial_{\alpha} D (r, \alpha)}{|D(r,\alpha)|^{2}} \left( c^{2}- \frac{\alpha^{2}}{r^{2}}\right) - \frac{2 \alpha}{r^{2} |D(r,\alpha)|} \right] dr. \\
 \end{aligned}
$$
If $c \rho_{1}$ and $c \rho_{2}$ are both larger than $\alpha$, the term  $c^{2}- \frac{\alpha^{2}}{r^{2}}$ is nonnegative.
Moreover $D(r, \alpha)$ and $\alpha$ have the same sign and $D$ is an non-decreasing function of $\alpha$. Thus the integral has the opposite sign as $ \alpha(\rho_{2}-\rho_{1})$. In particular, on any interval where $\alpha$ is positive and $g_{-}^{sub} \geq g_{+}^{sub}$ and on any interval where $\alpha$ is negative and $g_{-}^{sub} \leq g_{+}^{sub}$,
 $$   \left[ \frac{\partial}{\partial \alpha}F_{\alpha}  \right] (g_{-}^{sub}(\alpha))	- \left[ \frac{\partial}{\partial \alpha}F_{\alpha}  \right] (g_{+}^{sub}(\alpha)) \leq 0, $$
 which proves the lemma. 
\end{proof}
We are now in position to state the result in the subsonic case.
\begin{prop} \label{RSubSub}
Let $v \in \R$, $(\rho_{L},\alpha_{L}) \in \R_{+}^{*} \times \R$ and $(\rho_{R},\alpha_{R}) \in \R_{+}^{*} \times \R$ such that
 $$ -c \leq u_{L}-v \leq c \ \ \ \textrm{ and } \ \ \ -c \leq u_{R}-v \leq c. $$
 Then the set
 $$ \dis \mathcal{G}_{D}(v) \cap  \left( V_{-}(\rho_{L},\alpha_{L},v) \times V_{+}(\rho_{R},\alpha_{R},v) \right)$$
 is reduced to a unique element. Moreover, both traces belong to the subsonic triangle
 $$ \{ (\rho, \alpha), -c \rho \leq \alpha \leq c \rho \}= \{  (\rho,u), -c \leq u-v \leq c \}. $$
 In other words, the Riemann problem~\eqref{pbR} admits a unique solution, which is entirely subsonic.
\end{prop}
\begin{proof}
\emph{Step $1$: Properties of the traces.}

Suppose that
 $$\left( (\rho_{-},\alpha_{-}), (\rho_{+},\alpha_{+}) \right) \in \mathcal{G}_{D}(v) \cap  \left(V_{-}(\rho_{L},\alpha_{L},v) \times V_{+}(\rho_{R},\alpha_{R},v) \right) $$  then $\alpha_{-}=\alpha_{+}$. We denote by $\alpha$ this quantity. In Lemma~\ref{L:sub->sub} we proved that
 $$ -c \leq u_{-}-v \leq c \ \ \ \textrm{ and } \ \ \ -c \leq u_{+}-v \leq c. $$
Thus we can use Corollary~\ref{sub->sub} to obtain that
 \begin{itemize}
 \item If $\alpha=0$ then $\rho_{-}=\rho_{+}$ and $q_{-}=q_{+}$;
 \item If $\alpha>0$, then $\rho_{-}>\rho_{+}$;
 \item If $\alpha<0$, then $\rho_{-}<\rho_{+}$.
\end{itemize}

\emph{Step $2$: It exists a unique $\alpha_{0}$ such that $g_{-}^{sub}(\alpha_{0})=g_{+}^{sub}(\alpha_{0})$.}

As $c \rho_{L} \geq \alpha$, the upper extremity of $g_{-}^{sub}$ is the point $( \rho_{L,ex}, c \rho_{L,ex})$ (see the left of Figure~\ref{A1W}, $\rho_{L,ex}$ has been defined in Lemma~\ref{U-}). Similarly as $c \rho_{R} \geq |\alpha|$, the lower extremity of $g_{+}^{sub}$ is the point $( \rho_{R,ex}, -c \rho_{R,ex})$. If $( \rho_{L,ex}, c \rho_{L,ex})$ belongs to $\Gamma_{+}^{sub}$, as depicted on the left of Figure~\ref{FSubsonique},  we directly obtain the existence of $\alpha_{0}= c \rho_{L,ex}$. Similarly if $( \rho_{R,ex}, -c \rho_{R,ex})$ belong to $\Gamma_{+}^{sub}$ we have $\alpha_{0}= -c \rho_{R,ex}$.
\begin{psfrags}
\psfrag{a}{$\alpha$}
\psfrag{r}{$\rho$}
\psfrag{G-}{$\Gamma_{-}^{sub}$}
\psfrag{O-}{$\Omega_{-}^{sup}$} 
\psfrag{G+}{$\Gamma_{+}^{sub}$}
\psfrag{O+}{$\Omega_{+}^{sup}$} 
\psfrag{rL,qL}{$(\rho_{L},\alpha_{L})$}
\psfrag{a/c,a}{$(\rho_{0},\alpha_{0})$}
\psfrag{(ga,a)}{$(\rho_{0},\alpha_{0})$}
\psfrag{rR,qR}{$(\rho_{R},\alpha_{R})$}
\psfrag{a=cr}[][][1][26.5]{$\alpha=c \rho$}
\psfrag{a=-cr}[][][1][-23.5]{$\alpha=-c \rho$}
\begin{figure}[H]
\centering
\includegraphics[width=15cm]{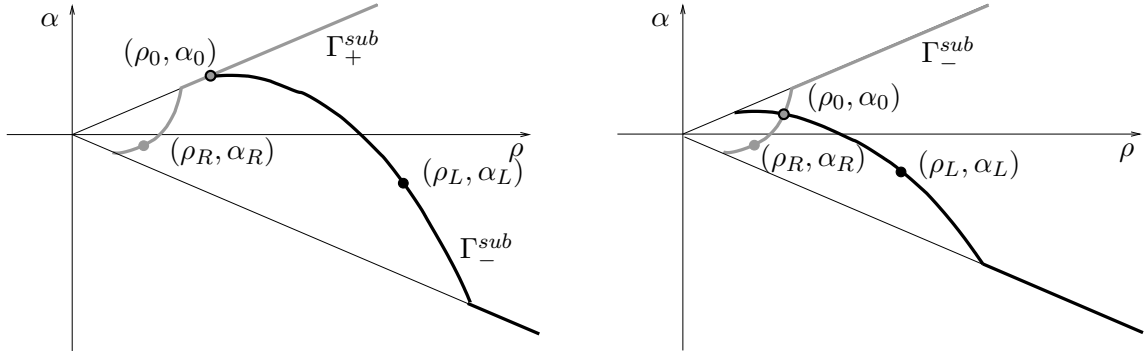}
\caption{ In black, the graph of $g_{-}^{sub}$, in grey the graph of $g_{+}^{sub}$. The intersection point $(\rho_{0}=g_{\pm}^{sub}(\alpha_{0}), \alpha_{0})$ is either on the sonic line $\alpha= c\rho$ or somewhere inside the triangle. }
 \label{FSubsonique}
 \end{figure}
\end{psfrags}
If we are not on one of these cases, as in the right of Figure~\ref{FSubsonique}, we have
$$ g_{-}^{sub}(c \rho_{L,ex})<g_{+}^{sub}(c \rho_{L,ex}) \ \ \ \textrm{ and }  \ \ \ g_{+}^{sub}(-c \rho_{R,ex})<g_{-}^{sub}(-c \rho_{R,ex}), $$
and $ \alpha_{0}$ exists  continuity of $g_{-}^{sub}-g_{+}^{sub}$. 

\vspace{3mm}

\emph{Step $3$: Conclusion.}
Let us suppose that  $\alpha_{0}>0$. According to the first step, the solution lies in the set $0 \leq \alpha \leq \alpha_{0}$ (we check on Figure~\ref{FSubsonique} that the relative positions of $g_{-}^{sub}(\alpha)$ and $g_{+}^{sub}(\alpha)$ are  only correct in that zone).  Moreover $\Delta(\alpha_{0})= 0$ and 
$$\dis \Delta(\alpha)=  \int_{g_{+}^{sub}(\alpha)}^{g_{-}^{sub}(\alpha)} \frac{1}{|D(\alpha, \rho)|} \left(c^{2}- \frac{\alpha^{2}}{r^{2}} \right) dr  \underset{\alpha \rightarrow 0_{+}}{\longrightarrow}+ \infty,$$ 
because $D(\rho, 0)=0$ and $g_{+}^{sub}(0) \neq g_{-}^{sub}(0)$. By Lemma~\ref{triangle}, the function $\Delta$ is monotonous, thus there exists a unique $\alpha$ in $[0, \alpha_{0}]$ such that $g_{+}^{sub}(\alpha) - g_{-}^{sub}(0)= 1$. Therefore, $(\rho_{-}, \alpha_{-})=(g_{-}^{sub}(\alpha), \alpha)$ and $(\rho_{+}, \alpha_{+})=(g_{+}^{sub}(\alpha), \alpha)$ verify Definition~\ref{defgermgen} and give a solution.
\end{proof}

\subsection{Resolution of the Riemann problem in the supersonic case}\label{pbR sup}
In this Section we prove Theorem~\ref{allfriction2} when $\alpha_{L}>c \rho_{L}$ or $\alpha_{R} <-c \rho_{R}$. Without loss of generality, let us assume  that $\alpha_{L}> c \rho_{L}$; the case $\alpha_{R}<- c \rho_{L}$ may be treated in a symmetrical way. Lemma~\ref{sub->sub} does not hold anymore. We must study in detail the case where  $(\rho_{L}, \alpha_{L})$ belongs to $V_{+}(\rho_{R},\alpha_{R}, v)$, which was excluded in the subsonic case. For this purpose, we introduce some notation, summarized on Figure~\ref{FSupSup}, which also recall the notation introduced in Lemmas~\ref{U-} and~\ref{U+}. In the sequel we denote by:
\begin{itemize}
\item for any subscript $i$ and for any point $(\rho_{i},\alpha_{i})$,  we denote by $(\tilde{\rho}_{i}:=\frac{\alpha_{i}^{2}}{c^{2} \rho_{i}}, \alpha_{i})$ the state reached with a shock at speed $v$. Remark that $ \dis \tilde{\tilde{\rho_{i}}}=\rho_{i}$;
 \item  $\rho_{L,ex}$ the extremity of the curve $g_{-}^{sub}$. Lemma~\ref{U-} shows that when $u_{L}-v>c$, $\rho_{L,ex}= \tilde{\rho_{L}}$ , and that 
 $$ \forall  \alpha< 0, \  (g_{-}^{sup}(\alpha), \alpha)=(\widetilde{g_{-}^{sub}(\alpha)}, \alpha) \ \textrm{ and } \ \forall  \alpha > 0, \ (g_{+}^{sup}(\alpha), \alpha)=(\widetilde{g_{+}^{sub}(\alpha)}, \alpha). $$ 
 \item $\rho_{E}= g_{+}^{sup}(\alpha_{L})$ the intersection of the line $\alpha=\alpha_{L}$ with $\Gamma_{+}^{sup}$. Note that $\dis \tilde{\rho_{E}} = g_{+}^{sub}(\alpha_{L})$. 
 \item $(\rho_{R,ex},\alpha_{R,ex})$ the extremity of the curve $g_{+}^{sub}$. We recall that, by Lemma~\ref{U+}, if $\alpha_{R} <-c \rho_{R}$, $\alpha_{R,ex}=\alpha_{R}$ and $\rho_{R,ex}= \frac{\alpha_{R}^{2}}{c^{2} \rho_{R}}$; and that if $\alpha_{R} \geq -c \rho_{R}$, $\rho_{R,ex} \leq \rho_{R}$ and $\alpha_{R,ex}=-c \rho_{R,ex}$. 
  \item $\rho_{F}= g_{-}^{sup}(\alpha_{R,ex})$ the intersection of the line $\alpha=\alpha_{R,ex}$ with $\Gamma_{-}^{sup}$. Note that $\dis \tilde{\rho_{F}} = g_{+}^{sub}(\alpha_{R,ex})$. 
\end{itemize}

\begin{psfrags}
 \psfrag{rho}{$\rho$}
 \psfrag{g+sub}{$g_{+}^{sub}$}
 \psfrag{g+sup}{$g_{+}^{sup}$}
 \psfrag{g-sub}{$g_{-}^{sub}$}
 \psfrag{g-sup}{$g_{-}^{sup}$}
 \psfrag{rl}{$\rho_L$}
 \psfrag{rlex}{$\tilde{\rho_{L}}=\rho_{L,ex}$}
 \psfrag{re}{$\rho_E$}
 \psfrag{tre}{$\tilde{\rho_E}$}
 \psfrag{rr}{$\rho_R$}
 \psfrag{rrex}{$\rho_{R,ex}$}
 \psfrag{rf}{$\rho_F$}
 \psfrag{trf}{$\tilde{\rho_F}$}
 \psfrag{a}{$\alpha$}
 \psfrag{al}{$\alpha_L$}
 \psfrag{arex}{$\alpha_{R,ex}$}
 \psfrag{u=c}{$u=c$}
 \psfrag{u=-c}{$u=-c$}
 \psfrag{O+}{$\Omega_{+}^{sup}$}
 \psfrag{O-}{$\Omega_{-}^{sup}$}
 \psfrag{(rr,ar)}{$(\rho_{R},\alpha_{R})$}
 \begin{figure}[H]
 \centering
 \includegraphics[width=15cm]{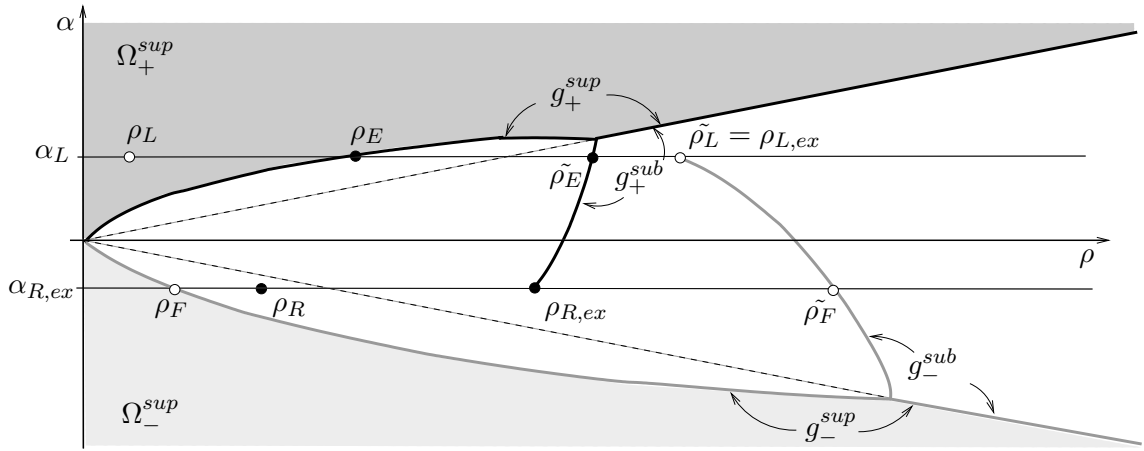}
 \caption{Notation for the supersonic case $\alpha_{L}>c \rho_{L}$. $V_{+}(v,\rho_{R},\alpha_{R})$ is the union of the open set $\Omega_{+}^{sub}$ above the graph of $g_{+}^{sup}$ and of the graph of $g_{+}^{sub}$;   $V_{-}(v,\rho_{L},\alpha_{L})$ is the union of the open set $\Omega_{-}^{sub}$  below the graph of $g_{-}^{sup}$ and of the graph of $g_{-}^{sub}$.}
 \label{FSupSup}
 \end{figure}
\end{psfrags}
We first exhibit a link between the position of $(\rho_{L}, \alpha_{L})$ and the position of $(\rho_{R}, \alpha_{R})$.
\begin{lemma} \label{restriction}
 Let $(\rho_{L},\alpha_{L}) \in \R_{+}^{*} \times \R$ and $(\rho_{R},\alpha_{R}) \in \R_{+}^{*} \times \R$ such that $\alpha_{L}> c \rho_{L}$ and $\alpha_{R}< -c \rho_{R}$. It is not possible to have
 $$(\rho_{L},\alpha_{L}) \in \Omega_{+}^{sup} \ \ \ \textrm{ and } \ \ \ (\rho_{R},\alpha_{R}) \in \Omega_{-}^{sup}.$$
\end{lemma}
\begin{proof} The hypothesis  $\alpha_{R}< -c \rho_{R}$ implies that $\rho_{R,ex}= \tilde{\rho_{R}}$. Suppose that $(\rho_{L},\alpha_{L})$ belongs to  $\Omega_{+}^{sup}$. Then $\rho_{L}< \rho_{E}$, and we have $\tilde{\rho_{E}}<\tilde{\rho_{L}}$. The monotonicity of $g_{+}^{sub}$ and $g_{-}^{sub}$ gives $\tilde{\rho_{R}}=\rho_{R,ex}<\tilde{\rho_{F}}$. Thus $\rho_{R}>\rho_{F}$, which means that $(\rho_{R},\alpha_{R})$ does not belong to $\Omega_{-}^{sup}$.
\end{proof}
It allows us to exclude the case $(\rho_{R},\alpha_{R}) \in \Omega_{-}^{sup}$ of our study. Indeed, if $(\rho_{R},\alpha_{R}) \in \Omega_{-}^{sup}$, then $\alpha_{R}< -c \rho$ and $(\rho_{L},\alpha_{L})$ does not belong to $\Omega_{-}^{sup}$, and we treat that case by symmetry. We now state the result in the supersonic case.
\begin{prop} \label{RSupSup}
Let $(\rho_{L},\alpha_{L}) \in \R_{+}^{*} \times  \R $ and $(\rho_{R},\alpha_{R}) \in \R_{+}^{*} \times  \R$ such that $\alpha_{L}>c \rho_{L}$ and $(\rho_{R},\alpha_{R}) \notin \Omega_{-}^{sup}$
Then the set
$$ \dis \mathcal{G}_{D}(\lambda) \cap  \left( V_{-}(\rho_{L},\alpha_{L},v) \times V_{+}(\rho_{R},\alpha_{R},v) \right)$$
is reduced to a unique element $((\rho_{-},\alpha_{-}),(\rho_{+},\alpha_{+}))$.
\end{prop}
\begin{proof}  We prove the result for a drag force only depend on $\alpha$, in which case we simply denote it by $D(\alpha)$. Because of the compatibility between the germ $\mathcal{G}_{D}(v)$ and the shocks at speed $v$, the germ is smaller and easier to describe. Instead of Definition~\eqref{defgermgen}, we can use the three relations of Corollary~\ref{sub->sub} and~\eqref{linearreformulation}. Even though it simplifies the proof, all the key ingredients are present in that case. The proof of Theorem~\ref{allfriction2} in the general case is given in Appendix~\ref{SFrictions}.

The proof relies on the relative positions of $\rho_{L}$ and $\rho_{E}$, and on the fact that the transformation $(\rho, \alpha) \mapsto \left( \tilde{\rho}= \frac{\alpha^{2}}{c^{2} \rho}, \alpha \right)$ reverses the positions of points around the point $(\alpha/c, \alpha)$, as depicted on Figure~\ref{FSupSup}.

\emph{Case $1$: $\rho_{L} \leq \rho_{E}$, or equivalently $(\rho_{L}, \alpha_{L}) \in \Omega_{+}^{sup}$. }

In that case, we can chose $(\rho_{-}, \alpha_{-})= (\rho_{L}, \alpha_{L})$ if and only if
$$ \left( \frac{\alpha_{L}^{2}}{\rho_{L}}+c^{2} \rho_{L} \right) -  \left( \frac{\alpha_{L}^{2}}{\rho_{E}}+c^{2} \rho_{E} \right) > D( \alpha_{L}).$$
Assume that this inequality is fulfilled. Then, as $ f_{\alpha_{L}}: \rho \mapsto \frac{\alpha_{L}^{2}}{\rho}+c^{2} \rho$ decreases on $(0, \frac{\alpha_{L}}{c})$, there exists a unique $\rho_{+} \in (\rho_{L}, \rho_{E})$ such that
$$ \left( \frac{\alpha_{L}^{2}}{\rho_{L}}+c^{2} \rho_{L} \right) -  \left( \frac{\alpha_{L}^{2}}{\rho_{+}}+c^{2} \rho_{+} \right) = D( \alpha_{L}).$$
Therefore we obtain the solution
$$ ((\rho_{L},\alpha_{L}), (\rho_{+},\alpha_{L})) \in \mathcal{G}_{D}(v) \cap  \left( V_{-}(\rho_{L},\alpha_{L},v) \times V_{+}(\rho_{R},\alpha_{R},v) \right).$$
In that case we have 
$$ \left( \frac{\alpha_{L}^{2}}{\tilde{\rho_{L}}}+c^{2} \tilde{\rho_{L}} \right) -  \left( \frac{\alpha_{L}^{2}}{\tilde{\rho_{E}}}+c^{2} \tilde{\rho_{E}} \right) > D( \alpha_{L}),$$
so $\Delta(\alpha_{L})>1$.  The function $g^{sub}_{-}-g^{sub}_{+}$ decreases, so it remain positive on $(0,\alpha_{L})$. Therefore, Lemma~\ref{L:sub->sub} shows that there is no other solution. Suppose now that 
 $$ \left( \frac{\alpha_{L}^{2}}{\rho_{L}}+c^{2} \rho_{L} \right) -  \left( \frac{\alpha_{L}^{2}}{\rho_{E}}+c^{2} \rho_{E} \right) < D( \alpha_{L}).$$
Thus we have that  $\Delta(\alpha_{L})<1$ and $\Delta(0)= + \infty$ and we conclude for the existence and the uniqueness as in Proposition~\ref{RSubSub}.
 
 \emph{Case $2$: $\rho_{L}  > \rho_{E}$, or equivalently $(\rho_{L}, \alpha_{L}) \notin \Omega_{+}^{sup}$. }

It implies that $\tilde{\rho_{E}} > \tilde{\rho_{L}}$. Moreover, $(\rho_{R}, \alpha_{R})$ does not belong to $\Omega_{-}^{sup}$. The fact that $\rho_{F}< \rho_{R}$ implies that $\tilde{\rho_{F}}\geq \rho_{R,ex}$ and that $\Delta(\alpha_{L})<0$. We conclude exactly as in the subsonic case, see the proof of Proposition~\ref{RSubSub}.
\end{proof}

\subsection{Asymptotics} \label{SAsymptotics}
Depending on the intensity of the drag force $D$, the model~\eqref{eqcouple} exhibit a whole range of behavior, from the lack of particle to the presence of a solid wall.
\begin{prop} Suppose that the drag force writes $D(\alpha)= \lambda D_{0}(\alpha)$, with $D_{0}$ a fixed drag force. When $\lambda$ tends to infinity, the solution of~\eqref{pbR} tends to the solution of the Riemann problem with the same initial data for the Euler equation with a solid wall along $x=vt$. When $\lambda$ vanishes, the solution of~\eqref{pbR} tends to the solution of the Riemann problem with the same initial data for the Euler equation without particle. 
\end{prop}
 
  \begin{figure}[H]
 \centering
 \includegraphics[width=13cm]{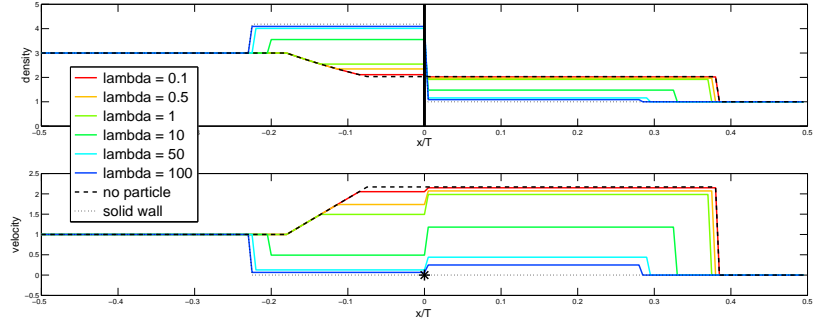}
 \caption{Exact solutions of the same Riemann problem for different friction coefficients $\lambda$}
 \label{FAsymSub}
  \end{figure} 
  
\begin{proof}
 Let us begin by the case $\lambda \longrightarrow + \infty$. Then in the supersonic case, it becomes impossible to chose $(\rho_{-}, \alpha_{-})= (\rho_{L}, \alpha_{L})$ because the inequality
 $$ \left( \frac{\alpha_{L}^{2}}{\rho_{L}}+c^{2} \rho_{L} \right) -  \left( \frac{\alpha_{L}^{2}}{\rho_{E}}+c^{2} \rho_{E} \right) > \lambda D_{0}(\alpha_{L})$$
 will always fail. This is illustrated by Figure~\ref{FAsymSup}, where the drag force is $D(\alpha)= \lambda \alpha$: this property holds for $\lambda=1$, while it is lost for $\lambda=20$. 
  \begin{figure}[h]
 \centering
 \includegraphics[width=13cm]{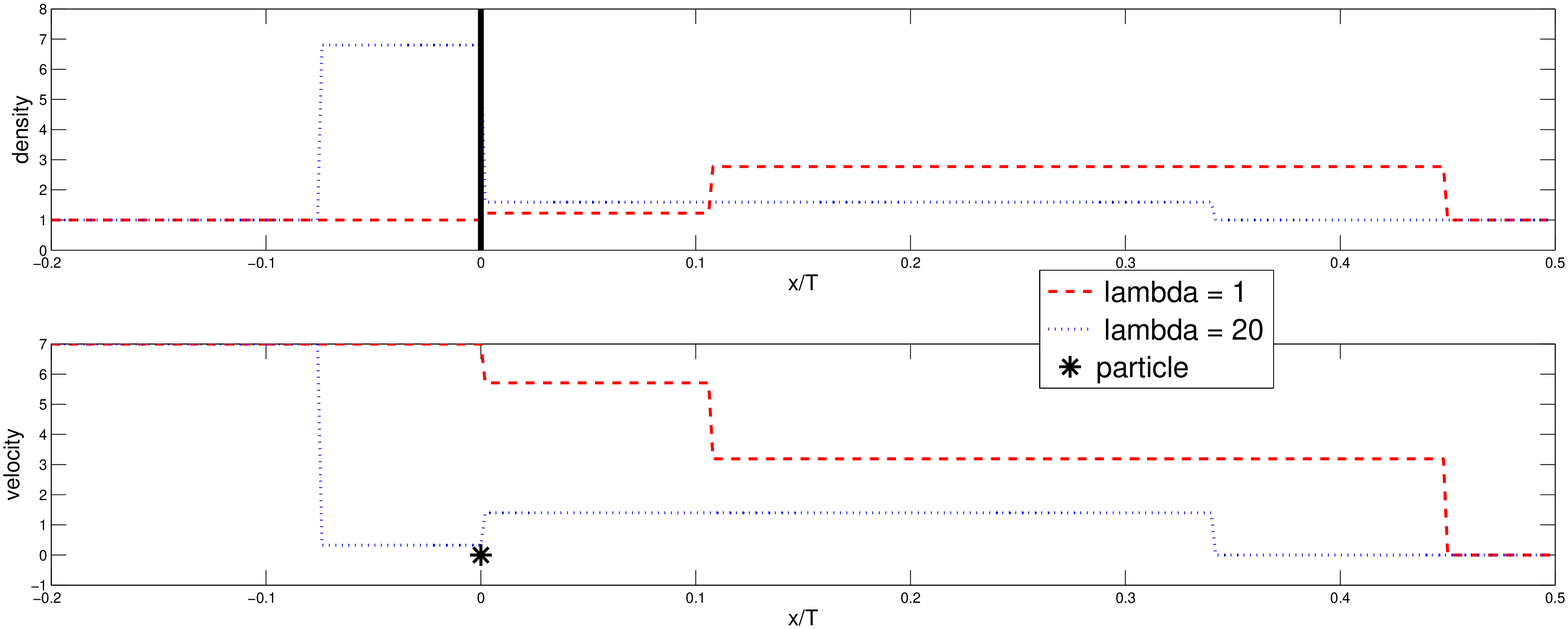}
 \caption{Two types of solutions when $\alpha_{L}> c \rho_{L}$, for the drag force $D(\alpha)= \lambda \alpha$. For small enough $\lambda$ there is no wave on the left of the particle. As $\lambda$ becomes greater, we recover the subsonic solution.}
 \label{FAsymSup}
 \end{figure} 
 In the subsonic case, as $g_{-}^{sub}$ and $g_{+}^{sub}$ take value in a bounded interval, for all non null $\alpha$, $\Delta (\alpha)$ tends to $0$ as $\lambda$ tends to infinity. Therefore, $ (\rho_{-}, \alpha_{-})$ tends to $(g_{-}^{sub}(0) ,0)$ while $(\rho_{+}, \alpha_{+})$ tends to $(g_{+}^{sub}(0) , 0)$. This is exactly the solution for the Riemann problem with a solid wall. When $\lambda$ tends to zero, remark that in the subsonic case (or the supersonic case when $(\rho_{L}, \alpha_{L})$ is not in $\Omega_{+}^{sup}$) the solution of $ \Delta(\alpha)=0$ tends to the crossing point $\alpha_{0}$. In the supersonic case when $(\rho_{L}, \alpha_{L}) \in \Omega_{+}^{sup}$, the inequality
$$ \left( \frac{\alpha_{L}^{2}}{\rho_{L}}+c^{2} \rho_{L} \right) -  \left( \frac{\alpha_{L}^{2}}{\rho_{E}}+c^{2} \rho_{E} \right) > \lambda D_{0}( \alpha_{L})$$
is verified for small enough $\lambda$, and $(\rho_{-},\alpha_{-})$ is $(\rho_{L}, \alpha_{L})$ while $(\rho_{+},\alpha_{+})$ tends to $(\rho_{L}, \alpha_{L})$. It corresponds to the value of the solution of the Euler equation without particle on the line $x=vt$. Those two asymptotics behaviors are depicted, in the subsonic case, on Figure~\ref{FAsymSub}.
\end{proof}

\section{Existence of up to three solutions when $|D|$ is not a decreasing function of $\rho$} \label{SMultipleSolutions}
When $|D|$ is not a decreasing function of $\rho$,  Condition~\eqref{croissance} may not hold and we can lose uniqueness and obtain up to three solutions. This is not surprising: the choice $D(\rho,u-v)= \rho$ is, up to a change of variable and pressure law, similar to the problem of the shallow water with a discontinuous topography, where these three solutions arise. In a more general framework, hyperbolic systems with resonant source term like ours have been investigate in~\cite{IT92} and~\cite{GLF04}, where the possible coexistence of three solutions is proved.

Suppose that this condition is reversed in
 \begin{equation} \label{decroissance}
 \forall \alpha, \ \forall \rho_{1}< \rho_{2} \leq \frac{|\alpha|}{c}, \ \ F_{\alpha}(\rho_{1})-F_{\alpha}(\rho_{2}) \geq F_{\alpha} (\tilde{\rho_{1}}) - F_{\alpha} (\tilde{\rho_{2}}). 
\end{equation}
Fix $\alpha_{L}>0$ and $\rho_{L}< \frac{\alpha_{L}}{c}$. For $0 \leq \theta_{1} \leq \theta_{2} \leq 1$, consider $\rho_{1}$ and $\rho_{2}$ such that
 $$F_{\alpha_{L}}(\rho_{L})- F_{\alpha_{L}}(\rho_{i})= \theta_{i}, $$
 and $\bar{\rho_{1}}$ and $\bar{\rho_{2}}$ such that
 $$F_{\alpha_{L}}(\tilde{\rho_{i}})- F_{\alpha_{L}}(\bar{\rho_{i}})= 1-\theta_{i}. $$
 We recall that $\tilde{\rho_{i}}= \frac{\alpha_{L}^{2}}{c^{2} \rho_{i}}$.
 Then we have
 $$
 \begin{aligned}
 F_{\alpha_{L}}(\bar{\rho_{1}})-F_{\alpha_{L}}(\bar{\rho_{2}}) &=  ( F_{\alpha_{L}}(\bar{\rho_{1}})-F_{\alpha_{L}}(\tilde{\rho_{1}}) )+ ( F_{\alpha_{L}}(\tilde{\rho_{1}})-F_{\alpha_{L}}(\tilde{\rho_{2}}) ) +  (F_{\alpha_{L}}(\tilde{\rho_{2}})-F_{\alpha_{L}}(\bar{\rho_{2}}) ) \\
 										&\leq -(1-\theta_{1})+ ( F_{\alpha_{L}}(\rho_{1})-F_{\alpha_{L}}(\rho_{2}) ) + (1-\theta_{2}) \ \ \textrm{ with~\eqref{decroissance}}\\
										& \leq 0.
\end{aligned}
 $$
 It makes possible the coexistence of three facts that exclude each other under Hypothesis~\eqref{croissance}.
\begin{itemize}
\item First,  there exists $(\rho_{0}, \alpha_{L}) \in \Omega_{+}^{sup}$ such that:
$$ F_{\alpha_{L}}(\rho_{L})-F_{\alpha_{L}}(\rho_{0}) = 1. $$
Thus $(\rho_{-}, \alpha_{-})= (\rho_{L},\alpha_{L})$ and $(\rho_{+}, \alpha_{+})= (\rho_{0},\alpha_{L})$ gives a solution that has no wave on the left of the particle, and two supersonic waves on its right, as depicted in Figure~\ref{3SolSupSup};
\item Second,
$$ F_{\alpha_{L}}(\tilde{\rho_{L}})-F_{\alpha_{L}}(\tilde{\rho_{E}}) \leq 1. $$
If $D$ is still an increasing function of $\alpha$, we can apply the proof of Proposition~\ref{RSubSub} to obtain the existence of a pair of subsonic traces. This solution has a $1$-wave on the left of the particle, and a $2$-wave on its right, as depicted in Figure~\ref{3SolSubSub};
\item Finally, the state reached on the right of the particle by jumping immediately ($\xi^{0}= -\eps/2$ or $\theta=0$ in Theorem~\ref{Tdefgermgen}) is smaller than $\tilde{\rho_{E}}$, while the state reached  by jumping in the end ($\xi^{0}= \eps/2$ or $\theta=1$) is larger than $\tilde{\rho_{E}}$. Then $(\rho_{-}, \alpha_{-})=(\rho_{I},\alpha_{L})$ and $(\rho_{+}, \alpha_{+})=(\tilde{\rho_{E}},\alpha_{L})$ are admissible traces around the particle. The corresponding solution has no wave on its left, and just a $2$-wave on its right, as depicted in Figure~\ref{3SolSupSub}.
\end{itemize}

 \begin{figure}[H]
 \centering
 \includegraphics[width=13cm]{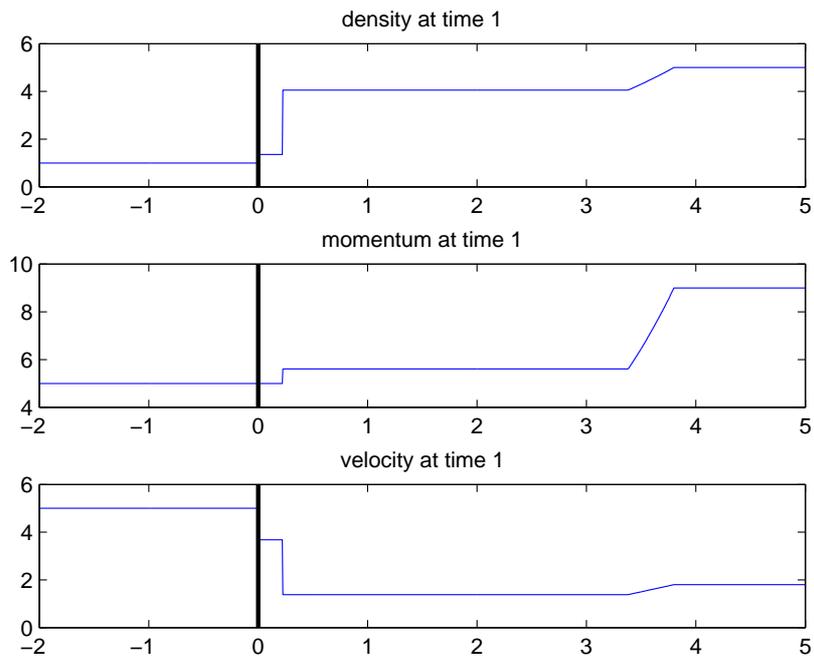}
 \caption{Solution with two supersonic traces}
 \label{3SolSupSup}
 \end{figure}
   \begin{figure}[H]
 \centering
 \includegraphics[width=13cm]{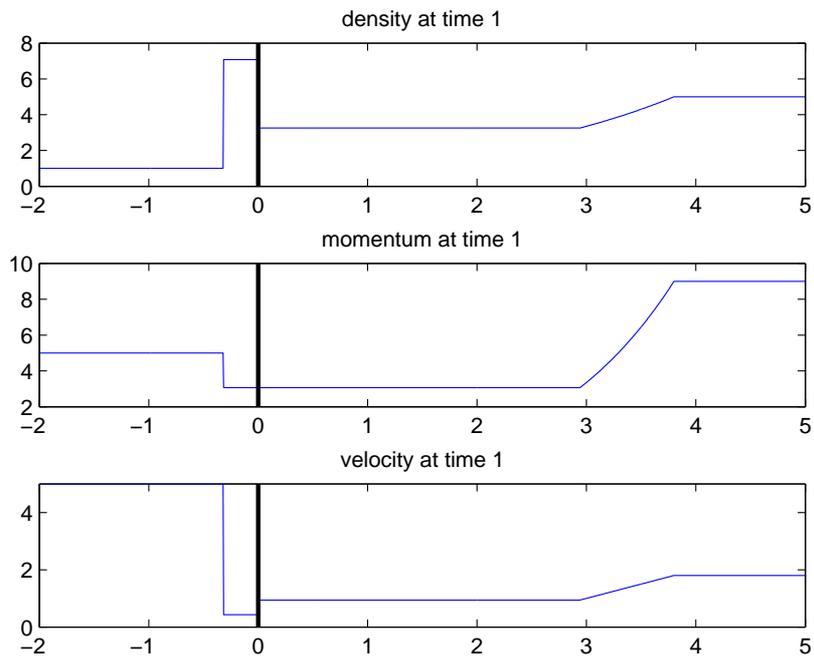}
 \caption{Solution with two subsonic traces}
 \label{3SolSubSub}
 \end{figure}
  \begin{figure}[H]
 \centering
 \includegraphics[width=13cm]{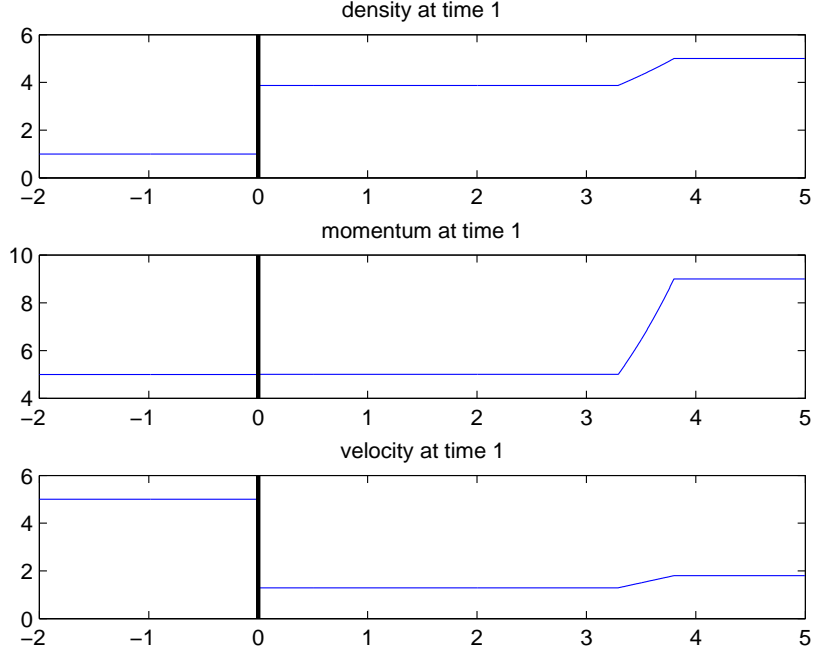}
 \caption{Solution with one supersonic trace (on the left) and one subsonic trace (on the right)}
 \label{3SolSupSub}
 \end{figure}
Figures~\ref{3SolSupSup},~\ref{3SolSupSub} and~\ref{3SolSubSub} represent the three exact solutions at time $T=1$ that are obtained for the Riemann problem
\begin{equation}
 \left\{
 \begin{array}{ll}
 \partial_t \rho + \partial_x q  = 0, \\
 \partial_t q + \partial_x \left( \frac{q^2}{\rho} +4 \rho \right) = -0.9 \rho^{2} u  \delta_{0}, \\
\rho(0,x) = 1 \ \mathbf{1}_{x<0} + 5 \  \mathbf{1}_{x>0}, \\
q(0,x) = 5 \ \mathbf{1}_{x<0} + 9 \ \mathbf{1}_{x>0}. 
\end{array} 
\right. 
\end{equation}
In that case we have $F_{\alpha}(\rho)= \frac{\alpha}{2\rho^{2}}+ \frac{c^{2}}{\alpha} \log(\rho)$. The traces have been numerically computed thanks to the previous analysis.
Another way to see that we can lose uniqueness, and obtain up to three solutions, is depicted on Figure~\ref{GLF}. Following~\cite{GLF04}, we introduce a merged $1$-wave, which regroups all the state $(\rho_{I},\alpha_{I})$ that can be reached from $(\rho_{L},\alpha_{L})$ through three successive steps:
\begin{itemize}
 \item From $(\rho_{L},\alpha_{L})$ we reach a state $(\rho_{-},\alpha_{-})$ on the left of the particle, by following a $1$-wave and a $2$-wave, both traveling at a speed smaller than the particle's velocity $v$:
 $$ (\rho_{-}, \alpha_{-}) \in V_{-}(\rho_{L},\alpha_{L},v);$$
 \item From  $(\rho_{-},\alpha_{-})$ we reach, through the particle, a state $(\rho_{+},\alpha_{+})$:
 $$((\rho_{-}, \alpha_{-}),(\rho_{+}, \alpha_{+}))\in \mathcal{G}_{D}(v); $$
 \item From $(\rho_{+},\alpha_{+})$ we reach $(\rho_{I},\alpha_{I})$ with a $1$-wave traveling faster than $v$.
\end{itemize}
The solutions of the Riemann problem~\eqref{pbR} are the intersections between this merged $1$-wave and the usual curve of $2$-waves arriving in $(\rho_{R},\alpha_{R})$. In the case of a supersonic left state the merged $1$-wave contains three different types of state:
\begin{itemize}
 \item Those obtained by taking $(\rho_{-},\alpha_{-})=(\rho_{L},\alpha_{L})$, then by decreasing continuously the quantity $F_{\alpha_{L}}$ of $1$ inside the particle. In that case $(\rho_{+}, \alpha_{+})$ is supersonic, and we can carry on with any $1$-rarefaction wave and some $1$-shocks to reach $(\rho_{I}, \alpha_{I})$. This is part $1$ of the black curve on Figures~\ref{GLF} and~\ref{GLFUnique}.
 \item Those obtained by taking $(\rho_{-},\alpha_{-})=(\rho_{L},\alpha_{L})$, then by decreasing continuously the quantity $F_{\alpha_{L}}$ of $ \theta$ for some $\theta \in [0, 1]$, making a shock inside the particle,  and finally continuously decreasing of  $(1- \theta) $ along $F_{\alpha_{L}}$. In that case $(\rho_{+},q_{+})$ is subsonic and there exists no $1$-wave faster than $v$ starting from $(\rho_{+},q_{+})$. This is part $2$ of the black curve on Figures~\ref{GLF} and~\ref{GLFUnique}.
 \item Those obtained by starting from $(\rho_{L},\alpha_{L})$ with a $1$-shock slower than the particle. There exists no such $1$-rarefaction wave and we reach a subsonic state $(\rho_{-},\alpha_{-})$, which lies on the dashed gray line on Figure~\ref{GLF}. The state $(\rho_{+},\alpha_{+})$ is necessarily obtained by decreasing continuously  of $1$ along the graph of $F_{\alpha_{-}}$. Therefore, $(\rho_{+},\alpha_{+})$ is subsonic and there exists no $1$-wave faster than $v$ starting from $(\rho_{+},q_{+})$. This is part $3$ of the black curve on Figures~\ref{GLF} and~\ref{GLFUnique}.
\end{itemize}
As we can see on Figures~\ref{GLF} and~\ref{GLFUnique}, the shape of the merged $1$-wave depends on the relative positions of the densities $\rho_{0,+}$ and $\rho_{1,+}$. The Hypothesis~\eqref{croissance} ensures that $\rho_{0,+} \leq \rho_{1,+}$, and the merged $1$-wave curve can be parametrized by $\rho$ as in Figure~\ref{GLFUnique}. If this hypothesis does not hold, it becomes possible that $\rho_{0,+} > \rho_{1,+}$, in which case the merged $1$-wave curve has a Z-shape and can intersect the $2$-waves curve up to three times as in Figure~\ref{GLF}.
\begin{rem}
 When $D$ depends only on $\alpha$, $\rho_{0,+}=\rho_{1,+}$ and the segment disappears. 
\end{rem}

\begin{psfrags} 
 \psfrag{c1}{$1$}
 \psfrag{c2}{$2$}
 \psfrag{c3}{$3$}
 \psfrag{r0+}{$\rho_{0,+}$}
 \psfrag{r1+}{$\rho_{1,+}$}
 \psfrag{(rl,al)}{$(\rho_{L}, \alpha_{L})$}
 \psfrag{(rr,ar)}{$(\rho_{R}, \alpha_{R})$}
 \begin{figure}[H]
 \centering
 \includegraphics[width=17cm]{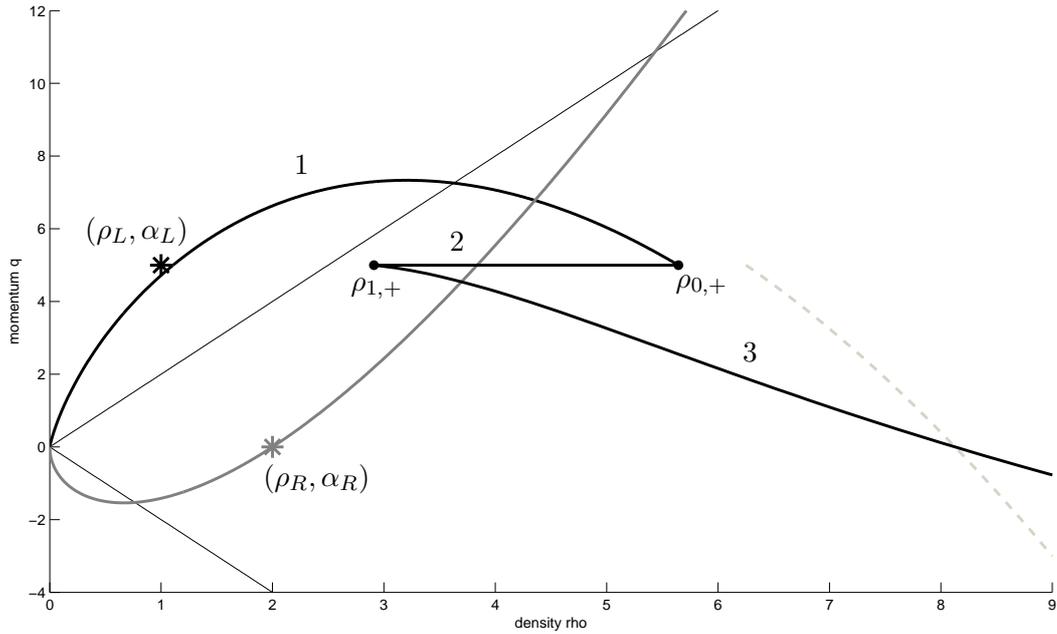}
 \caption{Example of non uniqueness with the friction $D(\rho,q,h')=\rho^{2}(u-h')= \rho \alpha$. The gray line is the usual curve of $2$-wave arriving in $(\rho_{R},\alpha_{R})$ (gray star). The black line is the merged $1$-curve from $(\rho_{L},\alpha_{L})$ (black star). }
 \label{GLF}
 \end{figure}
\end{psfrags}
\begin{psfrags} 
 \psfrag{c1}{$1$}
 \psfrag{c2}{$2$}
 \psfrag{c3}{$3$}
 \psfrag{r0+}{$\rho_{0,+}$}
 \psfrag{r1+}{$\rho_{1,+}$}
 \psfrag{(rl,al)}{$(\rho_{L}, \alpha_{L})$}
 \psfrag{(rr,ar)}{$(\rho_{R}, \alpha_{R})$}
 \begin{figure}[H]
 \centering
 \includegraphics[width=17cm]{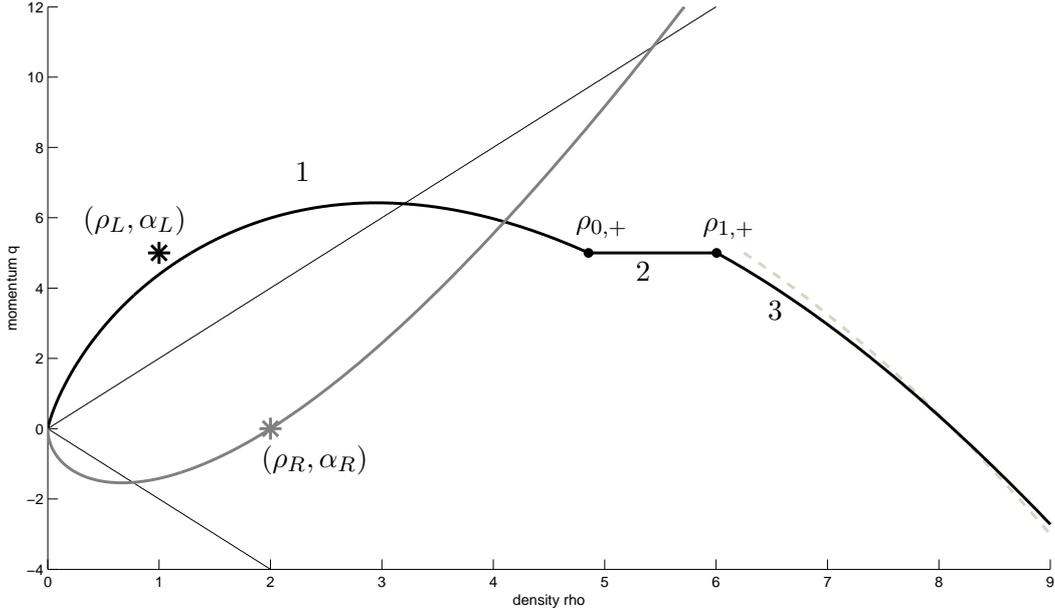}
 \caption{Example of uniqueness with the friction $\Gamma(\rho,q,h')= u-h' = \frac{\alpha}{\rho}$. The gray line is the usual curve of $2$-wave arriving in $(\rho_{R},\alpha_{R})$ (gray star). The black line is the merged $1$-curve from $(\rho_{L},\alpha_{L})$ (black star). }
 \label{GLFUnique}
 \end{figure}
\end{psfrags}

\begin{appendices}
\section{Proof of Proposition~\ref{RSupSup} when $D$ also depends on the density} \label{SFrictions}

When $D$ also depends on $\rho$, the germ is much larger than when it depends only of $\alpha$. For $(\rho_{-}, \alpha)$ is fixed, with $\alpha>0$ and $c \rho_{-}< \alpha$, it contains one supersonic state $(\rho_{0}, \alpha)$, and a whole set of subsonic states $(\rho, \alpha)$ with $\rho$ taking values in some interval $[\rho_{0+}, \rho_{1+}]$. When $D$ depends only on $\alpha$, this interval reduces to a single point $\rho_{0+} = \rho_{1+} = \tilde{\rho_{0}}$. In that case we only had to worry about the relative positions of $\rho_{0}$, $\rho_{E}$ and $\rho_{L}$ on the one hand, and of  $\tilde{\rho_{0}}$, $\tilde{\rho_{E}}$ and $\tilde{\rho_{L}}$ on the other hand (see Figure~\ref{FSupSup} for the notation). Those relative positions were easy to deduce from one another. In the general case, we have to study the relative positions of $\rho_{0}$, $\rho_{E}$ and $\rho_{L}$ on the one hand, and of the whole interval $[\rho_{0+}, \rho_{1+}]$ with $\tilde{\rho_{E}}$ and $\tilde{\rho_{L}}$ on the other hand. The following property insures that those relative are linked to each other nicely.
\begin{prop}
 if $|D|$ is a decreasing function of $\rho$, then for every $\alpha \in \R$, for all states $(\rho_{1}, \alpha)$ and $(\rho_{2}, \alpha)$ in $\R_{+}^{*}\times \R^{*}$ with $\rho_{1}<\rho_{2} \leq \frac{|\alpha|}{c}$ we have
\begin{equation} \label{croissance}
 F_{\alpha}(\rho_{1})-F_{\alpha}(\rho_{2}) \leq F_{\alpha} \left(\frac{\alpha^{2}}{c^{2} \rho_{1}} \right) - F_{\alpha} \left( \frac{\alpha^{2}}{c^{2} \rho_{2}} \right).
\end{equation}
\end{prop}

\begin{proof}
To prove~\eqref{croissance} it is sufficient to prove that  the function
 $$ \rho \mapsto F_{\alpha}(\rho)-F_{\alpha}\left(\frac{\alpha^{2}}{c^{2}\rho} \right) $$
 increases on $(0, \frac{|\alpha|}{c})$. We compute its derivative:
 $$ 
\begin{aligned}
 \frac{\partial}{\partial \rho} \left[F_{\alpha}(\rho)-F_{\alpha}\left(\frac{\alpha^{2}}{c^{2}\rho}\right)\right]
 	&= F_{\alpha}'(\rho)+ \frac{\alpha^{2}}{c^{2}\rho^{2}} F_{\alpha}'\left(\frac{\alpha^{2}}{c^{2}\rho}\right) \\
	&= \frac{1}{|D|(\rho, \alpha)} \left( c^{2}- \frac{\alpha^{2}}{\rho^{2}} \right) +  \frac{\frac{\alpha^{2}}{c^{2}\rho^{2}}}{|D|(\frac{\alpha^{2}}{c^{2}\rho}, \alpha)} \left( c^{2}- \frac{\alpha^{2}}{\frac{\alpha^{4}}{c^{4}\rho^{2}}} \right) \\
	&= \left[ \frac{1}{|D|(\rho, \alpha)}- \frac{1}{|D|(\frac{\alpha^{2}}{c^{2}\rho}, \alpha)} \right] \left( c^{2}- \frac{\alpha^{2}}{\rho^{2}} \right) .
\end{aligned}
$$
On $(0, \frac{|\alpha|}{c})$, this quantity has same sign as
$$|D|(\rho, \alpha)- |D| \left(\frac{\alpha^{2}}{c^{2}\rho}, \alpha \right),$$
which is positive if $|D|$ is a decreasing function of $\rho$, because $\rho \leq \frac{\alpha^{2}}{c^{2} \rho}$ on $(0, \frac{|\alpha|}{c})$.
\end{proof}

In the sequel we use the notation of Section~\ref{pbR sup} summarized on Figure~\ref{FSupSup}. We recall in particular that for any $(\rho, \alpha)$, the state denoting by $(\tilde{\rho}, \alpha)$ is reached with a shock at speed $v$: $\tilde{\rho}= \frac{\alpha^{2}}{c^{2}\rho}$. 
Let us first describe the form of the germ imposed by Hypothesis~\eqref{croissance}. 
\begin{lemma} \label{L3cases}
 Suppose that the drag force $D$ verifies the Hypothesis~\eqref{croissance}. Fix $(\rho_{-},q_{-})$ such that $\alpha_{-}=q_{-}- v \rho_{-}$ is positive and that $\rho_{-} \leq \frac{\alpha_{-}}{c}$. Then:
 \begin{itemize}
 \item If it exists $\rho_{0} \leq \frac{\alpha_{-}}{c}$ such that $F_{\alpha_{-}}(\rho_{-})- F_{\alpha_{-}}(\rho_{0}) = 1$, there exists a unique density $\rho_{1,+}$ greater than $\frac{\alpha_{-}}{c}$, such that $F_{\alpha_{-}}(\tilde{\rho_{-}})- F_{\alpha_{-}}(\rho_{1,+}) = 1$. Then solutions of~\eqref{eqregul} in $\xi= \eps/2$ can take the values $(\rho_{+},q_{+})$ with 
 $$\rho_{+} \in [\rho_{0,+}, \rho_{1,+}] \ \ \textrm{ and } \ \ q_{+}=\alpha_{-}+v \rho_{+},$$
where we denote by $\rho_{0,+}=\tilde{\rho_{0}}$;
 \item If $\rho_{0}$ does not exist but $\rho_{1,+}$ does,  the solutions of~\eqref{eqregul} in $\xi= \eps/2$ can take the values $(\rho_{+},q_{+})$ with 
  $$\rho_{+} \in \left[\frac{\alpha_{-}}{c}, \rho_{1,+} \right] \ \ \textrm{ and } \ \ q_{+}=\alpha_{-}+v \rho_{+},$$
and we denote by $\rho_{0,+}= \frac{\alpha_{-}}{c}$;
\item If neither $\rho_{0}$ nor $\rho_{1,+}$ exist, system~\eqref{eqregul} does not admit any solution on the whole interval $(-\eps/2, \eps/2)$.
\end{itemize}
\end{lemma}
Those three cases are illustrated in Figure~\ref{F3cases} below.
\begin{psfrags} 
 \psfrag{al}{$ 1$}
  \psfrag{rho}{$\rho$}
  \psfrag{2ac}{$ $}
  \psfrag{2ac+la}{}
  \psfrag{a/c}{ $\frac{\alpha_{-}}{c} $}
  \psfrag{rL}{ $\rho_{-} $}
  \psfrag{trL}{ $\tilde{\rho_{-}} $}
  \psfrag{r0}{ $\rho_{0} $}
  \psfrag{r0+}{ $\rho_{0,+} $}
  \psfrag{r1+}{ $\rho_{1,+} $}
  \psfrag{rho ->F}{}
 \begin{figure}[H] 
 \centering
 \includegraphics[width=18cm]{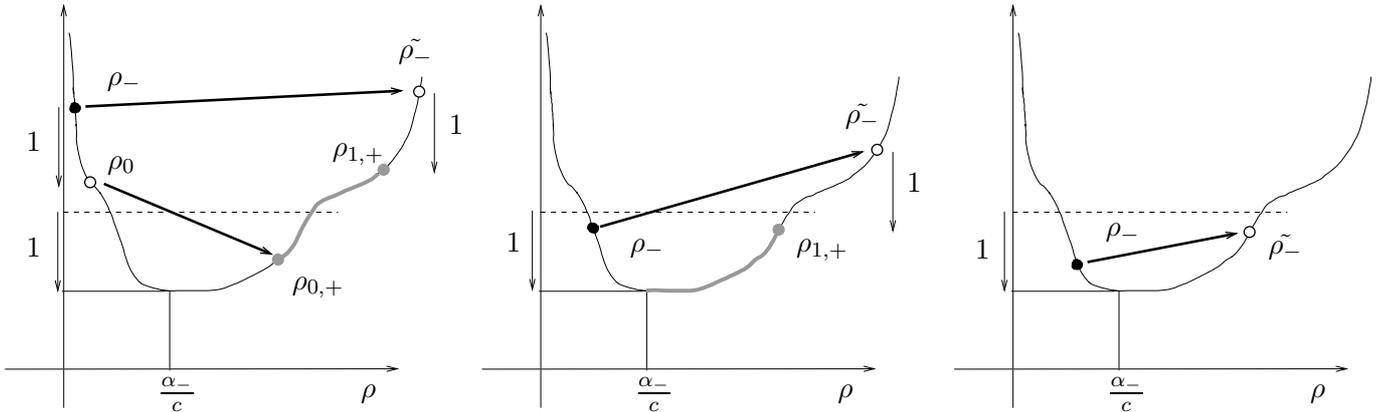}
 \caption{ The three cases of Lemma~\ref{L3cases} from left to right. The bold arrows are entropy shocks at speed~$v$.}
 \label{F3cases}
 \end{figure}
 \end{psfrags}
 \begin{proof}[Proof (Lemma~\ref{L3cases})]
 We first suppose that, as in top of Figure~\ref{FPos},
 $$ F_{\alpha_{L}}(\rho_{L})- F_{\alpha_{L}}(\rho_{E}) \geq1.$$
Then, there exists  $\rho_{0} \in [ \rho_{L}, \rho_{E}]$ such that
 $$ F_{\alpha_{L}}(\rho_{L})- F_{\alpha_{L}}(\rho_{0})= 1,$$
and the state $(\rho_{0}, \alpha_{L})$ belongs to $\Omega_{+}^{sup}$ and provides an admissible state on the right of the particle. Let us prove that it is the unique solution. As $\rho_{0}< \rho_{E}$, Hypothesis~\eqref{croissance} yields $\rho_{0,+}>\tilde{ \rho_{E}}$. We also proved in Lemma~\ref{L3cases} that $\rho_{1,+} \geq \rho_{0,+}$. Therefore, the interval $[\rho_{0,+}, \rho_{1,+}]$ does not intersect $\Gamma_{+}^{sub}$. Eventually, Hypothesis~\eqref{croissance} gives that $F_{\alpha}(\tilde{\rho_{L}})-F_{\alpha}(\tilde{\rho_{E}}) \geq 1$. Thus we cannot choose any subsonic traces by Lemma~\ref{triangle}. 
We now suppose that
$$ F_{\alpha_{L}}(\rho_{L})- F_{\alpha_{L}}(\rho_{E}) <1.$$ 
\begin{psfrags}
 \psfrag{rho}{$\rho$}
 \psfrag{V+}{$ \Omega_{+}^{sup}$} 
 \psfrag{V-}{$ \Omega_{-}^{sup}$} 
 \psfrag{rl}{$\rho_L$}
 \psfrag{r0}{$\rho_0$}
 \psfrag{r0+}{$\rho_{0,+}$}
 \psfrag{r1+}{$\rho_{1,+}$}
 \psfrag{rlex}{$\tilde{\rho_{L}}$}
 \psfrag{l a}{$\lambda \alpha_{L}$}
 \psfrag{re}{$\rho_E$}
 \psfrag{tre}{$\tilde{\rho_E}$}
 \psfrag{rr}{$\rho_R$}
 \psfrag{rrex}{$\rho_{R,ex}$}
 \psfrag{rf}{$\rho_F$}
 \psfrag{trf}{$\tilde{\rho_F}$}
 \psfrag{a}{$\alpha$}
 \psfrag{al}{$\alpha_L$}
 \psfrag{arex}{$\alpha_{R,ex}$}
 \psfrag{u=c}{$u=c$}
 \psfrag{u=-c}{$u=-c$}
 \psfrag{O+}{$\Omega_{+}^{sup}$}
 \psfrag{O-}{$\Omega_{-}^{sup}$}
 \psfrag{(rr,ar)}{$(\rho_{R},\alpha_{R})$}
 \begin{figure}[H]
 \centering
 \includegraphics[width=12cm]{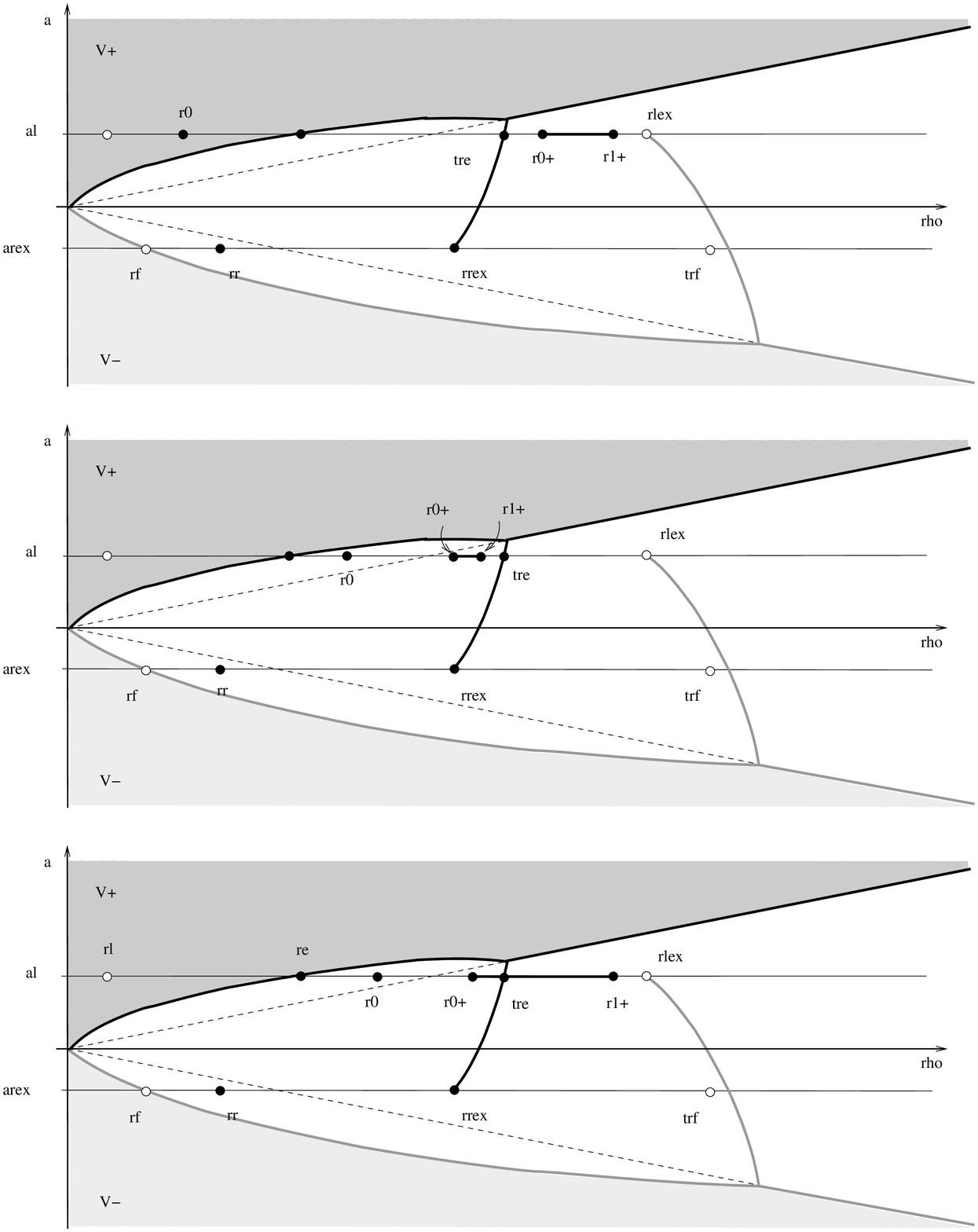}
 \caption{Relative positions of $\rho_{0,+}$ and $\rho_{1,+}$.}
 \label{FPos}
 \end{figure}
\end{psfrags}
 If $F_{\alpha}(\tilde{\rho_{L}})-F_{\alpha}(\tilde{\rho_{E}}) < 1$ (Figure~\ref{FPos}, in the middle), which rewrites $\Delta(\alpha_{L})<0$, Lemma~\ref{triangle} implies that there exists a unique solution inside the subsonic triangle $|\alpha|<c$. Moreover, Hypothesis~\eqref{croissance} yield $F_{\alpha}(\rho_{L})-F_{\alpha}(\rho_{E}) < 1 $, and $[\rho_{0,+}, \rho_{1,+}]$ is included in $[\frac{\alpha_{L}}{c}, \tilde{\rho_{E}}]$, so there is no solution with $(\rho_{L}, \alpha_{L})$ for the left trace. Eventually, if $F_{\alpha}(\tilde{\rho_{L}})-F_{\alpha}(\tilde{\rho_{E}}) > 1$ (at the bottom of Figure~\ref{FPos}), there is no solution in the subsonic triangle. But in that case, $\rho_{1,+}$ exists and is greater than $\tilde{\rho_{E}}$, while $\rho_{0,+}$ is smaller than $\tilde{\rho_{E}}$ (and might be equal to $\frac{\alpha_{L}}{c}$). Therefore we can take $(\rho_{L}, \alpha_{L})$ for the left trace and $(\tilde{\rho_{E}}, \alpha_{L})$ for the right trace.
\end{proof}

\end{appendices}
\phantomsection
\addcontentsline{toc}{section}{References} 
\nocite{*}
\bibliographystyle{alpha}
\bibliography{biblioM2}

\end{document}